\documentclass[11pt]{amsart}
\usepackage{amssymb}
\usepackage{mathtools}
\usepackage{color}

\newtheorem{theo}{Theorem} 

\newtheorem{lemma}{Lemma}[section]
\newtheorem{prop}[lemma]{Proposition}

\newtheorem{claim}[lemma]{Claim}

\newcommand{\vecc}{\overrightarrow}

\theoremstyle{remark}
\newtheorem{remark}[lemma]{Remark}

\theoremstyle{definition}

\newcommand{\cu}{\check{u}}
\newcommand{\cU}{\check{U}}
\newcommand{\ch}{\check{h}}
\newcommand{\cw}{\check{w}}

\newcommand{\lin}{\textsc{l}}

\newcommand{\NN}{\mathbb{N}}
\newcommand{\RR}{\mathbb{R}}

\newcommand{\eps}{\varepsilon}

\newcommand{\AAA}{\mathcal{A}}

\newcommand{\III}{\mathcal{I}}

\newcommand{\JJJ}{\mathcal{J}}

\newcommand{\PPP}{\mathcal{P}}

\newcommand{\tV}{\widetilde{V}}

\newcommand{\lf}{\left}
\newcommand{\rg}{\right}

\newcommand{\tih}{\tilde{h}}

\newcommand{\ttt}{\tilde{t}}

\newcommand{\tw}{\widetilde{w}}

\newcommand{\trho}{\tilde{\rho}}

\newcommand{\indic}{1\!\!1}

\newcommand{\tU}{\widetilde{U}}

\newcommand{\tu}{\widetilde{u}}

\newcommand{\loc}{\rm loc}
\newcommand{\hdot}{\dot{H}^1}

\newcommand{\EMPH}[1]{\medskip\noindent\textit{#1}.}

\DeclareMathOperator{\rad}{rad}

\DeclareMathOperator{\supp}{supp}

\numberwithin{equation}{section} 

\title[Radial critical wave equation]{Classification of radial solutions of the focusing, energy-critical wave equation}
\author[T.~Duyckaerts]{Thomas Duyckaerts$^1$}
\author[C.~Kenig]{Carlos Kenig$^2$}
\author[F.~Merle]{Frank Merle$^3$}
\thanks{$^1$LAGA, Universit\'e Paris 13 (UMR 7539). Partially supported by ERC Grant Dispeq and ERC advanced grant  no. 291214, BLOWDISOL}
\thanks{$^2$University of Chicago. Partially supported by NSF Grant DMS-0968472}
\thanks{$^3$Cergy-Pontoise (UMR 8088), IHES. Partially supported by ERC advanced grant  no. 291214, BLOWDISOL}

\date{\today}

\begin{document}
\begin{abstract}
In this paper, we describe the asymptotic behaviour of globally defined solutions and of bounded solutions blowing up in finite time of the radial energy-critical focusing non-linear wave equation in three space dimension. 
\end{abstract}

\maketitle

\section{Introduction and main result}
In this work we consider the energy-critical focusing non-linear wave equation in space dimension $3$:
\begin{equation}
\label{CP}
\left\{ 
\begin{gathered}
\partial_t^2 u -\Delta u-u^5=0,\quad (t,x)\in I\times \RR^3\\
u_{\restriction t=0}=u_0\in \hdot,\quad \partial_t u_{\restriction t=0}=u_1\in L^2,
\end{gathered}\right.
\end{equation}
where $I$ is an interval ($0\in I$), $u$ is real-valued, $\hdot:=\hdot(\RR^3)$, and $L^2:=L^2(\RR^3)$.

More precisely, we are interested in the so-called ``soliton resolution conjecture'' for radial solutions of \eqref{CP}. There has been a widespread belief in the mathematical physics community that, for large global solutions of dispersive equations, the evolution asymptotically decouples for large time into a sum of modulated solitons, a free radiation term and a term which goes to zero at infinity (see \cite{Miura76}, \cite{Segur73}, \cite{Schuur86BO}, \cite{Tao09BAMS}, \cite{Tao04DPDE}, \cite{IvancevicIvancevic10BO}). Such a result should hold for
globally well-posed equations (see \cite{Chatterjee2012P} for a recent result in this direction for mass-subcritical NLS), or in general  with the additional imposition that the solution does not
blow up. 
When blow-up may occur, such decompositions are always expected to be unstable, see Remark \ref{R:stability} below. So far, to the authors knowledge, the only cases when a result of this type are proved are for the integrable KdV for data with regularity and decay, due to Eckhaus and Schuur (see \cite{EcSc83}, \cite{Eckhaus86}) and for the integrable mKdV (see \cite{Schuur86BO}). Note that even the radial case of this conjecture is considered quite challenging (see \cite{Tao04DPDE} for example), and also that in the one dimensional case, only integrable models have been treated rigorously (see also \cite{ZakharovShabat71}, \cite{SegurAblowitz76} for heuristics in the case of the cubic NLS in one space dimension).

In the case of equation \eqref{CP}, since we are dealing with radial solutions, the solitons are of the form 
$\pm\frac{1}{\mu^{\frac{1}{2}}} W\left(\frac{x}{\mu}\right)$,  $ \mu>0$,
 where 
 $$W=\left(1+\frac{|x|^2}{3}\right)^{-\frac 12}$$ is the radial positive solution of $\Delta W+W^5=0$. 

Previous results for the equation \eqref{CP} dealt with solutions close to $W$ (see \cite{DuKeMe11a} for the radial case and \cite{DuKeMe10P} for the nonradial case) and for large, radial solutions in the case when the asymptotics hold along a specific sequence, and the solution is assumed to be bounded in norm \cite{DuKeMe11P}. Since we are in a critical case, there is another regime in which one expects a similar decomposition, that is for solutions which blow up in finite time, but with bounded critical norm. We also establish such a result in this paper for radial solutions of \eqref{CP}.

Results of this type for other equations in the global case include those for data close to the soliton for subcritical nonlinearities: see the works of Martel and Merle \cite{MaMe01} in the case of generalized KdV equations, of Buslaev and Perelman \cite{BuslaevPerelman92a,BuslaevPerelman92b} in the case of one dimensional NLS with specific nonlinearities, of Soffer and Weinstein \cite{SofferWeinstein90} for NLS with specific non-linearities in higher dimensions. For critical nonlinearities, see Martel and Merle for generalized KdV \cite{MaMe00}. For the finite time blow-up case in a critical setting, close to the soliton, we have the work of Martel and Merle \cite{MaMe02} for the critical generalized KdV, and of Rapha\"el and Merle \cite{MeRa04,MeRa05b} for the mass-critical NLS. 
In the finite time blow-up case, there are some large data results for critical equivariant wave-maps into the sphere due to Christodoulou\,-\,Tahvildar-Zadeh, Shatah\,-\,Tahvildar-Zadeh and Struwe (see \cite{ChTZ93,ShTZ97,Struwe02}) which show convergence along some sequence of times, locally to a soliton (harmonic map). In the case where a global Lyapunov functional is present in self-similar variables, results have been obtained for one dimensional wave equation in the work of Merle and Zaag \cite{MeZa08P}.
For formation of similar structures  (towering bubbles) for critical elliptic equations, for example on domains excluding a small ball, as the size of the ball goes to $0$, see the work of Musso and Pistoia \cite{MussoPistoia06} and references therein. 

We now turn to the description of our result.

We will restrict ourselves to the case of radial solutions, and denote by $r=|x|$ the radial variable. 
The equation \eqref{CP} is well-posed in $\hdot\times L^2$. We will  denote by $(T_-(u),T_+(u))$ the maximal interval of existence of $u$. On this interval of existence, the energy:
$$ E(u(t),\partial_tu(t))=\frac 12 \int |\nabla u(t,x)|^2\,dx+\frac{1}{2}\int (\partial_tu(t,x))^2\,dx-\frac 16 \int (u(t,x))^6\,dx$$
is conserved.

In all the paper, if $f$ and $g$ are two positive functions defined in a neighborhood of $\ell\in \RR\cup\{\pm\infty\}$, we will write 
$$f(t)\ll g(t) \text{ as }t\to \ell \text{ if and only if }
\lim_{t\to \ell}\frac{f(t)}{g(t)}=0.$$

\begin{theo}
\label{T:main}
 Let $u$ be a radial solution of \eqref{CP} and $T_+=T_+(u)$. Then one of the following holds:
\begin{itemize}
 \item \textbf{Type I blow-up:} $T_+<\infty$ and 
\begin{equation}
\label{BUp}
\lim_{t\to T_+} \|(u(t),\partial_tu(t)\|_{\hdot\times L^2}=+\infty. 
\end{equation} 
\item \textbf{Type II blow-up:} $T_+<\infty$ and there exist $(v_0,v_1)\in \hdot\times L^2$, an integer $J\in \NN\setminus\{0\}$, and for all $j\in \{1,\ldots,J\}$, a sign $\iota_j\in \{\pm 1\}$, and a positive function $\lambda_j(t)$ defined for $t$ close to $T_+$ such that
\begin{gather}
\label{hyp_lambda_bup}
 \lambda_1(t)\ll \lambda_2(t)\ll \ldots \ll\lambda_{J}(t)\ll T_+-t\text{ as }t\to T_+\\
\label{expansion_u_bup}
\lim_{t\to T_+}
\left\|(u(t),\partial_tu(t))-\left(v_{0}+\sum_{j=1}^J\frac{\iota_j}{\lambda_j^{1/2}(t)}W\left(\frac{x}{\lambda_{j}(t)}\right),v_{1}\right)\right\|_{\hdot\times L^2}=0.
\end{gather}
\item \textbf{Global solution:} $T_+=+\infty$ and there exist a solution $v_{\lin}$ of the linear wave equation, an integer $J\in \NN$, and for all $j\in \{1,\ldots,J\}$, a sign $\iota_j\in \{\pm 1\}$, and a positive function $\lambda_j(t)$ defined for large $t$ such that
\begin{gather}
\label{hyp_lambda}
 \lambda_1(t)\ll \lambda_2(t)\ll \ldots \ll\lambda_{J}(t)\ll t\text{ as }t\to +\infty\\
\label{expansion}
\lim_{t\to+\infty}
\left\|(u(t),\partial_tu(t))-\left(v_{\lin}(t)+\sum_{j=1}^J\frac{\iota_j}{\lambda_j^{1/2}(t)}W\left(\frac{x}{\lambda_{j}(t)}\right),\partial_tv_{\lin}(t)\right)\right\|_{\hdot\times L^2}=0.
\end{gather}
\end{itemize}
\end{theo}
\begin{remark}
\label{R:stability}
 It is known that the set $S_1$ of initial data $(u_0,u_1)\in \hdot\times L^2$ such that the corresponding solution of \eqref{CP} scatters forward in time is an open subset of $\hdot \times L^2$. It is widely believed that the set $S_2$ of initial data leading to a type I blow-up in positive time is also open  (see \cite{MatanoMerle09} for a similar result for the supercritical heat equation). Theorem \ref{T:main} says that any radial, finite energy solution of \eqref{CP} whose initial data is in the  complementary set $S_3$ of $S_1\cup S_2$ decouples in a finite sum of rescaled solitons and a radiation term. We believe that one could deduce from Theorem \ref{T:main}, using arguments similar to the ones in \cite{MatanoMerle11} for the radial heat equation or in \cite{MeRaSz10P} for the $L^2$-critical nonlinear Schr\"odinger equation, that $S_3$ is the boundary of $S_1\cup S_2$. We conjecture in particular that a nontrivial consequence of Theorem \ref{T:main} is that the asymptotic behaviour of solutions with initial data in $S_3$ is unstable.

For critical problems, understanding the boundary of the set of initial data  leading to blow-up is relevant. For example, for the $L^2$-critical NLS equation viewed as a limit of the Zakharov system,  the structurally stable blow-up is given by the pseudo-conformal blow-up, which is unstable with respect to the initial data (see \cite{Merle96b}).  See  also \cite{Gundlach99} and \cite{Bizon02} in the hyperbolic context.
\end{remark}

\begin{remark}
 In the finite time blow-up case, Theorem \ref{T:main} implies that 
$$\lim_{t\to T_+} \|(u(t),\partial_tu(t))\|_{\hdot\times L^2}=\ell\in \left(\|\nabla W\|_{L^2},+\infty\right)$$
exists. In particular, there is no oscillation of the norm, or mixed asymptotics, where the limit is infinite for one sequence $\{t_n\}\to T_+$ and finite for another sequence $\{t_n\}\to T_+$.
\end{remark}
\begin{remark}
 Another consequence of Theorem \ref{T:main} in the case $T_+<\infty$ is that solutions split into type I and type II blow-up. It is surprising that this can be established  in a critical problem outside the parabolic setting (see \cite{MatanoMerle09} for example).
\end{remark}
\begin{remark}
\label{R:existence}
 Note that in the case when $T_+<\infty$, both type I (see \cite[\S 6.2]{DuKeMe11P}) and type II (see \cite{KrScTa09}, and also \cite{HiRa10P}) exist. We expect that type II solutions with arbitrary $J\geq 1$ exist. For such constructions in the elliptic radial case, see for example \cite{MussoPistoia06}, and in the hyperbolic one-dimensional setting \cite{CoteZaag11P}.
\end{remark}
\begin{remark}
 For the case $T_+=+\infty$, Theorem \ref{T:main} implies that $\|(u(t),\partial_tu(t)\|_{\hdot\times L^2}$ is bounded on $[0,+\infty)$. More precisely,
$$\lim_{t\to+\infty} \|(u(t),\partial_tu(t))\|_{\hdot\times L^2}^2=\ell\quad \text{with }2E(u_0,u_1)\leq \ell \leq 3E(u_0,u_1).$$
(Note that unless $E(u_0,u_1)\geq 0$, $T_+<\infty$, see \cite{Levine74}, \cite{KeMe08}.) Thus there are no solutions such that $T_+=\infty$ and 
$$ \limsup_{t\to +\infty} \|(u(t),\partial_tu(t))\|_{\hdot\times L^2}=+\infty.$$
Such a result has been established before only in the dissipative case \cite{CaLi84} and for subcritical Klein-Gordon equations \cite{Cazenave85a}.

Solutions as in Theorem \ref{T:main} with $T_+=+\infty$ and $J=1$ have been recently constructed in \cite{DoKr12P}. As in Remark \ref{R:existence}, we expect that they exist for any $J\geq 0$ (the existence of wave operators, i.e the case $J=0$ is of course classical). In the case $T_+=+\infty$, Theorem \ref{T:main} implies that $ J\leq \frac{E(u_0,u_1)}{E(W,0)}.$
\end{remark}
The fundamental new ingredient of the proof is the following dispersive property that all radial solutions $u$ to \eqref{CP} (other than $0$ and $\pm W$ up to scaling) verify in their domain of definition, namely that there exist $R>0$, $\eta>0$ such that for all $t\geq 0$ or all $t\leq 0$, we have
\begin{equation}
 \label{*}
\int_{|x|>R+|t|} |\nabla u(t,x)|^2+(\partial_tu(t,x))^2\,dx\geq \eta.
\end{equation} 
This property is established using only the behavior of $u$ outside regions as in \eqref{*}, without using any global integral identities of virial type. (In fact this approach gives a new proof of the fact that $0$ and $\pm W$ are, up to scaling, the only $\hdot$ radial solutions of the elliptic equation $\Delta f+f^5=0$.) Using \eqref{*} with $R>0$, the finite speed of propagation and the profile decomposition of Bahouri and G\'erard \cite{BaGe99}, we are able to decouple the dynamics of different profiles in regions of the type in \eqref{*}. This is a fundamentally different approach to the one we used in \cite{DuKeMe11P}, which ultimately relied on virial identities. This new approach also yields a different proof of the characterization of radial solutions of \eqref{CP} with a compact trajectory up to scaling (see Theorem 2 of \cite{DuKeMe11a}) that does not rely on virial identities.

Let us emphasize that most of the proof of Theorem \ref{T:main} does not use any specific algebraic property of equation \eqref{CP}. In particular, the conservation of energy is only used in \S \ref{SS:boundedness} to show that the $\hdot\times L^2$ norm of a global solution does not go to infinity as $t\to+\infty$.
For this reason, we expect our general method to apply to many nonlinear dispersive equations and in particular to other hyperbolic problems, at least in the radial case. 
However, the deepest part of our paper, the characterization of solutions not satisfying \eqref{*} (carried out in Section \ref{S:channels}) is proved only in the context of equation \eqref{CP}.

In the nonradial case, even the elliptic equation $-\Delta u=u^5$ on $\RR^3$ is not well-understood yet (see \cite{Ding86} for the existence of solutions with infinitely many distinct energies), and we believe that for nonradial solutions of \eqref{CP}, only analogs of Theorem \ref{T:main} with some extra assumptions are within reach. The nonradial case seems very challenging and out of reach in its full generality for now. 

A key ingredient in our proof is the finite speed of propagation. 
However, in the case of infinite speed of propagation, the channel of energy method can still be applied (see e.g. \cite{MeRa08}).

The outline of the article is as follows. In Section \ref{S:channels}, we show the property \eqref{*} for nonstationary solutions of \eqref{CP}. In Section \ref{S:global} we prove Theorem \ref{T:main} in the case of global solutions. In Section \ref{S:blow-up}, we sketch the proofs in the finite-time blow-up case. 

%


\section{Existence of energy channels for nonstationary solutions}
\label{S:channels}
We will denote by $S(t)(v_0,v_1)$ the solution $v$ to the linear wave equation on $\RR\times \RR^3$:
\begin{equation}
 \label{lin_CP}
\left\{ 
\begin{gathered}
\partial_t^2 v -\Delta v=0,\quad (t,x)\in I\times \RR^3\\
v_{\restriction t=0}=v_0\in \hdot,\quad \partial_t v_{\restriction t=0}=v_1\in L^2.
\end{gathered}\right.
\end{equation} 
One can show (see \cite{DuKeMe11P} and Lemma \ref{L:lin} below) that if $v$ is not identically $0$, there exist $R>0$ and $\eta>0$ such that the following holds for all $t\geq 0$ or for all $t\leq 0$:
\begin{equation*}
\int_{|x|\geq R+|t|} |\nabla v(t,x)|^2+(\partial_t v(t,x))^2\,dx\geq \eta.
\end{equation*} 
In this section, we prove that essentially all radial solutions of the nonlinear equation \eqref{CP} satisfy this ``channel of energy" property in some sense, except the stationary solutions $0$ and $\frac{\pm 1}{\lambda^{1/2}}W\left(\frac{x}{\lambda}\right)$, $\lambda>0$.

If $(u_0,u_1)\in \hdot\times L^2$, we will denote by
\begin{equation}
 \label{def_rho}
\rho(u_0,u_1)=\inf \Big\{ r> 0,\text{ s.t. } \Big|\Big\{s>r,\; (u_0(s),u_1(s))\neq 0\Big\}\Big|=0\Big\},
\end{equation} 
where $\big|\cdot\big|$ denotes the Lebesgue measure.
We make the convention that $\rho(u_0,u_1)=+\infty$ if the set over which we take the infimum is empty. The main results of this section are the following:
\begin{prop}
\label{P:channel1}
 Let $u$ be a non-zero, radial solution of \eqref{CP} such that for all $\lambda>0$ and for all signs $+$ or $-$, $\left(u_0\pm \frac{1}{\lambda^{1/2}}W\left(\frac{x}{\lambda}\right),u_1\right)$ is not compactly supported. Then there exist constants $R>0$, $\eta>0$ and a global, radial solution $\tu$ of \eqref{CP}, with initial data $(\tu_0,\tu_1)$, scattering in both time directions such that
\begin{equation}
 \label{tu0=u0}
\left(\tu_0(r),\tu_1(r)\right)=(u_0(r),u_1(r))\text{ for } r>R,
\end{equation} 
and the following holds for all $t\geq 0$ or for all $t\leq 0$:
\begin{equation}
\label{channel1}
\int_{|x|> R+|t|} |\nabla \tu(t,x)|^2+(\partial_t \tu(t,x))^2\,dx \geq \eta.
\end{equation}
\end{prop}

\begin{prop}
 \label{P:W+compct}
Let $R_0>0$ be a large constant. Then the following properties hold.

Let $u$ be a radial solution of \eqref{CP} such that $(h_0,h_1):= (u_0\pm W,u_1)$ is compactly supported and not identically $0$. Then:
\begin{enumerate}
 \item \label{I:small_supp} There exists a solution $\cu$, defined for $t\in [-R_0,R_0]$, and $R'\in (0, \rho(h_0,h_1))$ such that
\begin{equation}
 \label{cu0=u0}
(\cu_0(r),\cu_1(r))=(u_0(r),u_1(r))\text{ for } r>R',
\end{equation}
and the following holds for all $t\in [0,R_0]$ or for all $t\in [-R_0,0]$:
\begin{equation}
\label{propa_support}
\rho\big(\cu(t)\pm W,\partial_t\cu(t)\big)=\rho(h_0,h_1)+|t| 
\end{equation} 
\item \label{I:big_supp} Assume furthermore that $\rho(h_0,h_1)>R_0$. Let $R<\rho(h_0,h_1)$ be close to $\rho(h_0,h_1)$. Then there exists $\eta>0$ and a global, radial solution $\tu$ of \eqref{CP}, scattering in both time directions, such that \eqref{tu0=u0} holds, and \eqref{channel1} is satisfied for all $t\geq 0$ or for all $t\leq 0$.
\end{enumerate}
\end{prop}

Let us mention that Propositions \ref{P:channel1} and \ref{P:W+compct} generalize Theorem 2 of \cite{DuKeMe11a} (for the case $N=3$), which states  that any radial solution of \eqref{CP} which has a relatively compact trajectory up to scaling in $\hdot\times L^2$ is a stationary solution.

The proofs of Propositions \ref{P:channel1} and \ref{P:W+compct} are based on dispersive properties of radial solutions to the linear wave equation (see Lemma \ref{L:lin} below), the small data theory of \eqref{CP} and related equations, and refined localization arguments based on finite speed of propagation. Note that we never use in the proofs of the propositions any variational characterization of $W$ or any uniqueness result on the elliptic equation $-\Delta u=u^5$: the fact that $0$, $W$ and $-W$ are (up to scaling) the only radial finite-energy solutions of this equation on $\RR^3$ can be seen as consequences of Propositions \ref{P:channel1} and \ref{P:W+compct}.

\subsection{Preliminaries}
\label{SS:preliminaries}
We start with a few notations. We will denote by $\vec{u}=(u,\partial_t u)$. 

Let $(u_0,u_1)\in \hdot\times L^2$ radial, and $R>0$. We define $(\tu_0,\tu_1)=\Psi_R(u_0,u_1)$ by
\begin{align*}
 \tu_0(r)=u_0(r)&,\quad \tu_1(r)= u_1(r) &\text{ for }r\geq R\\
\tu_0(r)=u_0(R)&,\quad \tu_1(r)=0&\text{ for }r<R.
\end{align*}
Note that $(\tu_0,\tu_1)\in \hdot\times L^2$, that $(u_0(r),u_1(r))=(\tu_0(r),\tu_1(r))$ for $r>R$, and 
\begin{equation*}
 \left\|(\tu_0,\tu_1)\right\|_{\hdot\times L^2}^2= \int_{|x|>R} \left(|\nabla u_0|^2+u_1^2\right)\,dx.
\end{equation*} 
We will denote by $D_x^{1/2}$ the Fourier multiplier with symbol $|\xi|^{1/2}$. 

We recall the following Lemma on radial, linear solutions, proved in \cite{DuKeMe11a}:
\begin{lemma}
\label{L:lin}
 Let $R>0$, $(u_0,u_1)\in \hdot \times L^2$ (radial) and $u_{\lin}=S(t)(u_0,u_1)$. Then the following holds for all $t\geq 0$ or for all $t\leq 0$:
$$ \int_{R+|t|}^{+\infty} \big[\partial_r (ru_{\lin}(t,r))\big]^2+\big[\partial_t (ru_{\lin}(t,r))\big]^2\,dr\geq \frac 12 \int_{R}^{+\infty} \big[\partial_r (r u_0(r))\big]^2+\big[ r u_1(r)\big]^2\,dr.$$ 
\end{lemma}
The norm in Lemma \ref{L:lin} and the usual $\hdot$ norm are related by the following formula, given by a straightforward integration by parts: for any radial $f\in \hdot$ and $R>0$,
\begin{equation}
 \label{IPP1}
\int_{R}^{+\infty} \left(\partial_r(rf)\right)^2\,dr=\int_{R}^{+\infty} (\partial_r f)^2r^2\,dr-R f^2(R).
\end{equation}

We will also need the following small data Cauchy problem result:
\begin{lemma}
\label{L:CPCarlos}
 There exists a small $\delta_0>0$ with the following property. Let $I$ be an interval with $0\in I$. Let $V=V(t,x)\in L^8(I\times \RR^3)$. Assume
\begin{multline}
\label{V_small}
\left\|V\right\|_{L^8(I\times \RR^3)}+\left\|D_x^{1/2} V\right\|_{L^4(I\times \RR^3)}+\left\|D_x^{1/2} V^2\right\|_{L^{\frac 83}(I\times \RR^3)}\\
+\left\|D_x^{1/2} V^3\right\|_{L^2(I\times \RR^3)}+\left\|D_x^{1/2} V^4\right\|_{L^{\frac 85}(I\times \RR^3)}<\delta_0
\end{multline}
and consider $(h_0,h_1)\in \hdot\times L^2$ such that
\begin{equation}
 \label{h0_small}
\|(h_0,h_1)\|_{\hdot\times L^2}\leq \delta_0.
\end{equation} 
Then there exists a unique solution $h$ of 
\begin{equation}
\label{CP_h}
\left\{ 
\begin{gathered}
\partial_t^2 h -\Delta h=5V^4h+10V^3h^2+10V^2h^3+5Vh^4+h^5,\quad (t,x)\in I\times \RR^3\\
h_{\restriction t=0}=h_0,\quad \partial_t h_{\restriction t=0}=h_1,
\end{gathered}\right.
\end{equation}
with $\vec{h}\in C^0\left(I,\hdot\times L^2\right)$, and $h\in L^8(I\times \RR^3)$. Furthermore, letting $h_{\lin}(t)=S(t)(h_0,h_1)$,
\begin{equation}
 \label{h_almost_lin}
\sup_{t\in I}\left\|\vec{h}(t)-\vec{h}_{\lin}(t)\right\|_{\hdot\times L^2} \leq \frac{1}{10}\left\|(h_0,h_1)\right\|_{\hdot\times L^2}.
\end{equation} 
\end{lemma}
We will use Lemma \ref{L:CPCarlos} with two choices of $V$, given by the following claim:
\begin{claim}
 \label{C:WV_OK}
\begin{enumerate}
 \item\label{I:W_OK}
Assume $V(t,x)=W(x)$. Then there exists a small $t_0>0$ such that \eqref{V_small} holds with $I=(-2t_0,2t_0)$.
\item\label{I:V_OK}
Let $R_0>0$ and define $V(t,x)$ as
\begin{equation}
 \label{def_V}
\left\{
\begin{aligned}
V(t,x)&=W(x) \text{ if } |x|\geq R_0+|t|\\
V(t,x)&=W(R_0+|t|)\text{ if }|x|< R_0+|t|.
\end{aligned}
\right.
\end{equation}
Then if $R_0$ is large, \eqref{V_small} holds with $I=\RR$.
\end{enumerate}
\end{claim}
We will prove Lemma \ref{L:CPCarlos} and Claim \ref{C:WV_OK} in Appendix \ref{A:linearized}.

We conclude this preliminary subsection with the following elementary claim that will be needed throughout the proofs:

\begin{claim}
 \label{C:limsup}
Let $\tu$ be a global solution of \eqref{CP} such that for some $R>0$,
\begin{equation}
\label{limsup}
 \limsup_{t\to +\infty} \int_{|x|>R+|t|} |\nabla \tu(t,x)|^2+(\partial_t \tu(t,x))^2\,dx>0,
\end{equation} 
then \eqref{channel1} holds for some $\eta>0$ and all $t\geq 0$. An analoguous statement holds for negative times. 
\end{claim}
\begin{proof}
Indeed, assume \eqref{limsup}, and (by contradiction) that there exists a sequence $t_n\to \infty$ such that
\begin{equation*}
 \lim_{n\to +\infty} \int_{|x|>R+|t_n|} |\nabla \tu(t_n,x)|^2+(\partial_t \tu(t_n,x))^2\,dx=0.
\end{equation*} 
Let $u_n$ be the solution of \eqref{CP} such that 
$$(u_n(t_n),\partial_tu_n(t_n))=\Psi_{R+|t_n|}(\tu(t_n),\partial_t\tu(t_n)).$$
Then
$$\lim_{n\to\infty} \|(u_n(t_n),\partial_tu_n(t_n))\|_{\hdot\times L^2}=0.$$
Consider a small $\eps>0$ and let $n$ such that $\|(u_n(t_n),\partial_t u_n(t_n))\|_{\hdot\times L^2}<\eps$. By the small data theory, $u_n$ is globally defined and for all $t$,
$$\|(u_n(t),\partial_tu_n(t))\|_{\hdot\times L^2}<2\eps.$$
By finite speed of propagation, for all $t$,
$$ (\tu(t_n+t,r),\partial_t\tu(t_n+t,r))=(u_n(t_n+t,r),\partial_tu_n(t_n+t,r))\text{ if } r> R+t_n+|t|,$$
and hence
$$ \limsup_{t\to+\infty} \int_{|x|\geq R+t} |\nabla \tu(t,x)|^2+(\partial_t\tu(t,x))^2\,dx<2\eps.$$
As $\eps>0$ is arbitrarily small, this contradicts \eqref{limsup}, concluding the proof.
\end{proof}

 \subsection{Proof of the channel energy property}
This subsection is dedicated to the proofs. We start by showing Proposition \ref{P:W+compct} (\S \ref{SS:small_supp} and \ref{SS:cpct_supp}), then prove Proposition \ref{P:channel1} (see \S \ref{SS:other_cases}).

\subsubsection{Propagation of the support for a compactly supported perturbation of $W$}
\label{SS:small_supp}
In this subsection, we prove point \eqref{I:small_supp} of Proposition \ref{P:W+compct}. We divide the proof into three steps.

\EMPH{Step 1: linearization around $W$} 
By our assumptions (up to a sign change), $(u_0,u_1)=(W,0)+(h_0,h_1)$ where $(h_0,h_1)$ is compactly supported. Using that $W$ is globally defined, we get that there exists $\eps>0$ such that for any solution $U$ of \eqref{CP} with $\left\|(W,0)-(U,\partial_tU)_{\restriction t=0}\right\|_{\hdot\times L^2}<\eps$, we have $[-R_0,+R_0]\subset I_{\max}(U)$. 

We let $(\ch_0,\ch_1)=\Psi_{R'}(h_0,h_1)$, where $R'<\rho(h_0,h_1)$ is chosen close to $\rho(h_0,h_1)$, so that
$$0<\|(\ch_0,\ch_1)\|_{\hdot\times L^2}<\eps.$$
Let $\cu$ be the solution of \eqref{CP} with initial data $(W+\ch_0,\ch_1)$. Equivalently, $\ch=\cu-W$ is the solution of
\begin{equation}
\label{CP_ch}
\left\{ 
\begin{gathered}
\partial_t^2 \ch -\Delta \ch=5W^4\ch+10W^3\ch^2+10W^2\ch^3+5W\ch^4+\ch^5,\quad (t,x)\in \RR\times \RR^3\\
(\ch,\partial_t \ch)_{\restriction t=0}=(\ch_0,\ch_1).
\end{gathered}\right.
\end{equation}
By the definition of $\eps$, $\cu$ and $\ch$ are defined on $[-R_0,R_0]$. By finite speed of propagation, $(\cu,\partial_t\cu)=(W,0)$ for $r\geq \rho(h_0,h_1)+|t|$, and thus
\begin{equation}
\label{bound_rho_above}
 \rho(\ch(t),\partial_t\ch(t))\leq \rho(h_0,h_1)+|t|,\quad t\in [-R_0,R_0].
\end{equation} 
We must show that for all $t\in [-R_0,0]$ or for all $t\in [0,R_0]$,
\begin{equation}
\label{bound_rho_below}
 \rho(\ch(t),\partial_t\ch(t))= \rho(h_0,h_1)+|t|,\quad t\in [-R_0,R_0].
\end{equation} 
\EMPH{Step 2: small time interval}

By Claim \ref{C:WV_OK}, there exists a small $t_0>0$ such that $W$ satisfies the assumption \eqref{V_small} of Lemma \ref{L:CPCarlos} with $I=[-t_0,t_0]$. We show in this step that \eqref{bound_rho_below} holds for all $t\in [-t_0,0]$ or all $t\in [0,t_0]$. 

Let $\rho_0$ close to $\rho(h_0,h_1)$ such that $R'<\rho_0<\rho(h_0,h_1)$, and define
$$(g_0,g_1)=\Psi_{\rho_0}(\ch_0,\ch_1).$$
If $\rho(h_0,h_1)-\rho_0$ is small enough, we can assume
$$ \|(g_0,g_1)\|_{\hdot\times L^2}\leq \delta_0,$$
where $\delta_0$ is given by Lemma \ref{L:CPCarlos}. By Lemma \ref{L:CPCarlos}, there exists a unique solution $g$ of
\begin{equation}
\label{CP_g2}
\left\{ 
\begin{gathered}
\partial_t^2 g -\Delta g=5W^4g+10W^3g^2+10W^2g^3+5Wg^4+g^5,\quad (t,x)\in [-t_0,t_0]\times \RR^3\\
(g,\partial_t g)_{\restriction t=0}=(g_0,g_1),
\end{gathered}\right.
\end{equation}
and denoting by $g_{\lin}(t)=S(t)(g_0,g_1)$,
\begin{equation}
\label{vecg}
\sup_{-t_0<t<t_0} \|\vec{g}(t)-\vec{g}_{\lin}(t)\|_{\hdot\times L^2}\leq \frac{1}{10}\|(g_0,g_1)\|_{\hdot\times L^2}. 
\end{equation} 
By Lemma \ref{L:lin} and formula \eqref{IPP1}, the following holds for all $t\in [0,t_0]$ or for all $t\in [-t_0,0]$:
\begin{multline}
\label{interm_g0}
\int_{|x|\geq \rho_0+|t|} \left(|\nabla g_{\lin}(t,x)|^2+(\partial_t g_{\lin}(t,x))^2\right)\,dx\geq \int_{\rho_0+|t|}^{+\infty} \left(\partial_r (rg_{\lin}(t,r))\right)^2+\left(\partial_t (r g_{\lin}(t,r))\right)^2\,dr\\
 \geq \frac 12\int_{\rho_0}^{+\infty} \left(\partial_r (rg_0)\right)^2+(r g_{1})^2\,dr=
\frac 12\int_{|x|\geq \rho_0}\left(|\nabla g_0|^2+g_{1}^2\right)\,dx-\frac 12\rho_0(g_0(\rho_0))^2.
\end{multline}
We have
\begin{multline}
\label{IPPg_0}
 |g_0(\rho_0)|= \left|\int_{\rho_0} ^{\rho(h_0,h_1)} \partial_r g_0(r)\,dr\right|\leq \sqrt{\Big(\rho(h_0,h_1)-\rho_0\Big)\int_{\rho_0}^{\rho(h_0,h_1)}(\partial_r g_0(r))^2\,dr}\\
\leq \frac{\sqrt{\rho(h_0,h_1)-\rho_0}}{\rho_0}\sqrt{\int_{\rho_0}^{\rho(h_0,h_1)}(\partial_r g_0(r))^2r^2\,dr},
\end{multline}
and thus if $\rho_0$ is close enough to $\rho(h_0,h_1)$, $\rho_0|g_0(\rho_0)|^2\leq \frac{1}{4}\|\nabla g_0\|_{L^2}^2$. Combining with \eqref{vecg} \eqref{interm_g0} we get that the following holds for all $t\geq 0$ or for all $t\leq 0$:
\begin{equation}
 \label{channel_g}
\int_{|x|\geq \rho_0+|t|} \left(|\nabla g(t,x)|^2+(\partial_t g(t,x))^2\right)\,dx\geq \frac{1}{40}\int \left(|\nabla g_0|^2+g_1^2\right)\,dx>0.
\end{equation}
By finite speed of propagation (see the argument after \eqref{tih=g} below), one can replace $g$ by $\ch$ in the left-hand side of \eqref{channel_g}. Hence
$$ \rho(\ch(t),\partial_t\ch(t))\geq \rho_0+|t|$$ 
for all $t\in [-t_0,0]$ or for all $t\in [0,t_0]$. Letting $\rho_0\to \rho(h_0,h_1)$, we get (in view of \eqref{bound_rho_above}) that \eqref{bound_rho_below} holds on $[-t_0,0]$ or on $[0,t_0]$, concluding this step.

\EMPH{Step 3: end of the proof} 

It is now easy to conclude by an induction argument. Assume to fix ideas that \eqref{bound_rho_below} holds for all $t\in [0,t_0]$. Applying Step 2, to $t\to \ch(t+t_0)$, we get that the following holds for all $t\in [0,t_0]$ or for all $t\in [t_0,\min(2t_0,R_0)]$: 
\begin{equation}
\label{supp_induc}
 \rho\left(\ch(t_0+t),\partial_t\ch(t_0+t)\right)=\rho(h_0,h_1)+|t_0|+|t|.
\end{equation} 
If \eqref{supp_induc} holds on $[0,t_0]$, we get a contradiction with the fact that \eqref{bound_rho_below} holds at $t=0$. Thus \eqref{supp_induc} holds on $[0,\min(2t_0,R_0)]$. Arguing inductively, we get that \eqref{bound_rho_below} holds on $[0,R_0]$.

\subsubsection{Compactly supported perturbation of $W$ with large support}
\label{SS:cpct_supp}
In this part, we prove case \eqref{I:big_supp} of Proposition \ref{P:W+compct}. Let $u$ be a radial solution of \eqref{CP} such that 
$$(u_0,u_1)=(W,0)+(h_0,h_1)\text{ and }R_0<\rho(h_0,h_1)<\infty,$$
where the large parameter $R_0>0$ is given by Claim \ref{C:WV_OK}.

Define $V(t,x)$ by \eqref{def_V}. Let $(g_0,g_1)=\Psi_R(h_0,h_1)$ where $R\in \big(R_0,\rho(h_0,h_1)\big)$. We chose $R$ close to $\rho(h_0,h_1)$, so that
$$ \|(\nabla g_0,g_1)\|_{\hdot\times L^2}\leq \delta_0,$$
where $\delta_0$ is given by Lemma \ref{L:CPCarlos}. By Lemma \ref{L:CPCarlos}, there exists a unique solution $g$ of
\begin{equation}
\label{CP_g}
\left\{ 
\begin{gathered}
\partial_t^2 g -\Delta g=5V^4g+10V^3g^2+10V^2g^3+5Vg^4+g^5,\quad (t,x)\in \RR\times \RR^3\\
g_{\restriction t=0}=g_0,\quad \partial_t g_{\restriction t=0}=g_1.
\end{gathered}\right.
\end{equation}
Furthermore, letting $g_{\lin}=S(t)(g_0,g_1)$, we have
\begin{equation}
 \label{g_gL}
\sup_{t\in \RR} \left\|\vec{g}_{\lin}(t)-\vec{g}(t)\right\|_{\hdot\times L^2} \leq \frac{1}{10}\left\|(g_0,g_1)\right\|_{\hdot\times L^2}.
\end{equation}
We divide the proof into two steps.
 
\EMPH{Step 1} In this step, we show that the following holds for all $t\geq 0$ or for all $t\leq 0$:
\begin{equation}
 \label{bound_g}
\int_{|x|\geq R+|t|} \left(|\nabla g(t,x)|^2+(\partial_t g(t,x))^2\right)\,dx\geq \frac{1}{40} \|(g_0,g_1)\|^2_{\hdot\times L^2}>0.
\end{equation} 
Indeed by Lemma \ref{L:lin} and the integration by parts formula \eqref{IPP1}, the following holds for all $t\geq 0$ or for all $t\leq 0$:
\begin{multline*}
 \int_{|x|\geq R+|t|} \left(|\nabla g_{\lin}(t,x)|^2+(\partial_t g_{\lin}(t,x))^2\right)\,dx\geq \int_{R+|t|}^{+\infty} \big((\partial_r(r g_{\lin}))^2+(\partial_t(r g_{\lin}))^2\big)\,dr\\
\geq \frac{1}{2}\int_R^{+\infty} \left((\partial_r(rg_0))^2+(rg_1)^2\right)\,dr=\frac{1}{2}\left(\|\nabla g_0\|^2_{L^2}+\|g_1\|_{L^2}^2-R g_0^2(R)\right).
\end{multline*}
By \eqref{IPPg_0} (with $R$ instead of $\rho_0$), and using that $\rho(h_0,h_1)=\rho(g_0,g_1)$, we get that if $R$ is close enough to $\rho(h_0,h_1)$, $R|g_0(R)|^2\leq \frac{1}{4}\|\nabla g_0\|_{L^2}^2$, which shows \eqref{bound_g} in view of \eqref{g_gL}.

\EMPH{Step 2: conclusion of the proof}

Let $\tu$ be the solution of \eqref{CP} with initial data $(\tu_0,\tu_1)=\Psi_R(u_0,u_1)=\Psi_R(W+h_0,h_1)$. Let $\tih=\tu-W$. If $R_0$ is chosen large and $R>R_0$ close enough to $\rho(h_0,h_1)$, it is easy to see that  $\|(\tu_0,\tu_1)\|_{\hdot\times L^2}$ is small, and thus that $\tu$ is globally defined and scatters in both time directions. Moreover, $\tih$ satisfies
\begin{equation}
\label{CP_th}
\left\{ 
\begin{gathered}
\partial_t^2 \tih -\Delta \tih=5W^4\tih+10W^3\tih^2+10W^2\tih^3+5W\tih^4+\tih^5,\quad (t,x)\in \RR\times \RR^3\\
(\tih(t,r),\partial \tih(t,r))_{\restriction t=0}=(h_0(r),h_1(r))=(g_0(r),g_1(r))\text{ if }r>R.
\end{gathered}\right.
\end{equation}
Using that $W=V$ if $|x|>R_0+|t|$ we get by finite speed of propagation and the equations \eqref{CP_g} and \eqref{CP_th},
\begin{equation}
\label{tih=g}
(\tih,\partial_t\tih)(t,r)=(g,\partial_tg)(t,r)\text{ for } r>R+|t|.
\end{equation} 
Indeed, $w= \tih-g$ satisfies the equation:
$$\partial_t^2 w-\Delta w - Mw =F$$
where 
$$M(t,x)=5W^4+10W^3(\tih+g)+10W^2(\tih^2+g\tih+g^2)+5W\sum_{k=0}^3 \tih^{3-k}g^k+\sum_{k=0}^4\tih^{4-k}g^k$$ 
and 
$$F(t,x)=5(V^4-W^4)g+10(V^3-W^3)g^2+10(V^2-W^2)g^3+4(V-W)g^4.$$
One can check that for any compact interval $I\subset \RR$, $M\in L^8(I\times \RR^3)$, $D_x^{1/2}M\in L^4(I\times \RR^3)$ and $D_x^{1/2}F\in L^{4/3}(I\times \RR^3)$.
Moreover, $F(t,x)=0$ for $|x| > R_0 +|t|$, and
$(w,\partial_tw)_{\restriction t=0}= 0$ for $|x|> R_0$.

The solution $w$ can be constructed by a fixed point on small time intervals as in Appendix \ref{A:linearized}. Writing  the solution $w$ iteratively via Duhamel formula, one shows using the finite speed of propagation for the free wave propagator that $w=0$ for $|x|>R_0+|t|$, which gives \eqref{tih=g}. We omit the details.

By Step 1, we deduce that the following holds for all $t\geq 0$ or for all $t\leq 0$:
\begin{equation}
\label{channel_th}
 \int_{|x|\geq R+|t|} \left(|\nabla \tih(t)|^2+(\partial_t \tih(t))^2\right)\,dx\geq \eta,
\end{equation} 
where $\eta=\frac{1}{40}\|(g_0,g_1)\|^2_{\hdot\times L^2}>0$. Using that 
$$\lim_{t\to \pm \infty} \int_{|x|> R+|t|} |\nabla W|^2\,dx =0,$$
we get that one of the following holds at least for one sign $+$ or $-$:
$$\limsup_{t\to \pm \infty} \int_{|x|>R+|t|} \left(|\nabla \tu(t)|^2+(\partial_t \tu(t))^2\right)\,dx\geq \eta,$$
which concludes the proof of case \eqref{I:big_supp} of Proposition \ref{P:W+compct}, in view of Claim \ref{C:limsup}.
\subsubsection{Other solutions}
\label{SS:other_cases}
In this part we prove Proposition \ref{P:channel1} as a consequence of the following lemma:
\begin{lemma}
 \label{L:channel1}
Let $u$ be a global, radial solution of \eqref{CP} such that, for some $R>0$,
\begin{multline}
 \label{no_channel}
\lim_{t\to +\infty}\int_{|x|>R+|t|} |\nabla u(t,x)|^2+(\partial_tu(t,x))^2\,dx\\
=\lim_{t\to -\infty} \int_{|x|>R+|t|} |\nabla u(t,x)|^2+(\partial_tu(t,x))^2\,dx=0. 
\end{multline} 
Then $(u_0,u_1)$ is compactly supported, or there exists $\lambda>0$ and $\iota\in \{\pm 1\}$ such that
$$ (u_0,u_1)-\left(\frac{\iota}{\lambda^{1/2}}W\left(\frac{x}{\lambda}\right),0\right)$$
is compactly supported.
\end{lemma}
\begin{proof}[End of the  proof of Proposition \ref{P:channel1}]
We first assume Lemma \ref{L:channel1} and prove Proposition \ref{P:channel1}. Let $u$ be a radial solution of \eqref{CP}. 

We first note that the conclusion of Proposition \ref{P:channel1} holds when $(u_0,u_1)$ is compactly supported.
Indeed, in this case, the proof of \S \ref{SS:cpct_supp} remains valid, replacing  $V$ and $W$ by $0$, and using the standard small data Cauchy theory for equation \eqref{CP} instead of Lemma \ref{L:CPCarlos}.
We note that this case was treated in \cite{DuKeMe11P} (see Lemma 3.4).

Assume that $(u_0,u_1)$ is not compactly supported, and let $(\tu_0,\tu_1)=\Psi_R(u_0,u_1)$, where $R>0$ is chosen large, so that ($\eps>0$ is given by the small data Cauchy theory for \eqref{CP}):
$$ 0<\|(\tu_0,\tu_1)\|_{\hdot\times L^2}<\eps.$$
Let $\tu$ be the solution of \eqref{CP} with initial data $(\tu_0,\tu_1)$. According to Claim \ref{C:limsup}, there exists $\eta>0$ such that $\tu$ satisfies \eqref{channel1} for all $t\geq 0$ or all $t\leq 0$ unless:
\begin{multline}
 \label{no_channel_bis}
\lim_{t\to +\infty}\int_{|x|>R+|t|} |\nabla \tu(t,x)|^2+(\partial_t\tu(t,x))^2\,dx\\=\lim_{t\to -\infty} \int_{|x|>R+|t|} |\nabla \tu(t,x)|^2+(\partial_t\tu(t,x))^2\,dx=0. 
\end{multline}
Assume \eqref{no_channel_bis}. Then by Lemma \ref{L:channel1}, $(\tu_0,\tu_1)$ is compactly supported or there exists $\lambda>0$ and $\iota\in \{\pm 1\}$ such that $(\tu_0,\tu_1)-\left(\frac{\iota}{\lambda^{1/2}}W\left(\frac{x}{\lambda}\right),0\right)$ is compactly supported. In the first case, $(u_0,u_1)$ is compactly supported, which is already excluded. In the second case, $(u_0,u_1)-\left(\frac{\iota}{\lambda^{1/2}}W\left(\frac{x}{\lambda}\right),0\right)$ is compactly supported, contradicting the assumptions of Proposition \ref{P:channel1} and concluding the proof. 
\end{proof}
It remains to prove Lemma \ref{L:channel1}. Let $u$ be as in Lemma \ref{L:channel1}. We let $v=ru$, $v_0=ru_0$ and $v_1=ru_1$. We first show two Lemmas.
\begin{lemma}
 \label{L:funct_eq}
There exists a constant $C_0>0$ (not depending on $u$) such if for some $r_0>0$
\begin{equation}
 \label{assu_r0}
\int_{r_0}^{+\infty} \left((\partial_ru_0)^2+u_1^2\right)r^2\,dr\leq \delta_0,
\end{equation} 
where $\delta_0>0$ is small, then
\begin{equation}
 \label{funct_eq}
\int_{r_0}^{+\infty} \Big((\partial_r v_0)^2+v_1^2\Big)\,dr\leq C_0\frac{|v_0(r_0)|^{10}}{r_0^5}.
\end{equation} 
Furthermore, for all $r,r'$ with $r_0\leq r\leq r'\leq 2r$,
\begin{equation}
 \label{bound_v0}
\left|v_0(r)-v_0(r')\right|\leq \sqrt{C_0}\frac{|v_0(r)|^5}{r^2} \leq \sqrt{C_0}\delta_0^2|v_0(r)|.
\end{equation}
\end{lemma}
\begin{proof}
 We first assume \eqref{funct_eq} and prove \eqref{bound_v0}. If $r_0\leq r\leq r'\leq 2r$, we have by \eqref{funct_eq}:
\begin{equation*}
 |v_0(r)-v_0(r')|\leq \left|\int_r^{r'} \partial_r v_0(\sigma)\,d\sigma\right|\leq \sqrt{r}\sqrt{\int_{r}^{+\infty}(\partial_r v_0(\sigma))^2\,d\sigma}\leq \sqrt{C_0r}\, \frac{|v_0(r)|^5}{r^{5/2}},
\end{equation*}
hence the first inequality in \eqref{bound_v0}. If $r\geq r_0$, then (see formula \ref{IPP1}),
$$\frac{1}{r} v_0^2(r)=ru_0^2(r)\leq \int_r^{+\infty} (\partial_r u_0(\sigma))^2\sigma^2\,d\sigma\leq \delta_0,$$
which yields the second inequality in \eqref{bound_v0}.

We next show \eqref{funct_eq}. Let $u_{\lin}(t,r)=S(t)(u_0,u_1)$ and $v_{\lin}=r u_{\lin}$. By Lemma \ref{L:lin}, the following holds for all $t\geq 0$ or for all $t\leq 0$
\begin{equation}
 \label{bound_vL}
\int_{r_0+|t|}^{+\infty} (\partial_r v_{\lin}(t,r))^2+(\partial_t v_{\lin}(t,r))^2\,dr\geq \frac 12 \int_{r_0}^{+\infty} (\partial_r v_0(r))^2+v_1^2(r)\,dr.
\end{equation}  
Recall the definition of $\Psi_R$ from the beginning of Subsection \ref{SS:preliminaries}.
Let $(\tu_0,\tu_1)=\Psi_{r_0}(u_0,u_1)$, $\tu_{\lin}=S(t)(\tu_0,\tu_1)$, and $\tu$ the solution of \eqref{CP} with initial data $(\tu_0,\tu_1)$. By assumption \eqref{assu_r0}, $\|(\tu_0,\tu_1)\|_{\hdot\times L^2}^2\leq \delta_0$. Taking $\delta_0$ small, we get by the small data Cauchy theory that for all $t\in \RR$,
\begin{multline*}
 \left\| (\tu-\tu_{\lin},\partial_t \tu-\partial_t\tu_{\lin})(t)\right\|_{\hdot\times L^2}\leq C\|(\tu_0,\tu_1)\|_{\hdot\times L^2}^5= C\left(\int_{r_0}^{+\infty} \left((\partial_r u_0)^2+u_1^2\right)r^2\,dr\right)^{5/2}\\
=C\left(\int_{r_0}^{+\infty} \left((\partial_r v_0)^2+v_1^2\right)\,dr+r_0 u_0^2(r_0)\right)^{5/2}.
\end{multline*}
Hence
\begin{multline*}
 \int_{r_0+|t|}^{+\infty} \left(\left(\partial_r \tu_{\lin}(t)\right)^2+\left(\partial_t \tu_{\lin}(t)\right)^2\right)r^2\,dr
\\
\leq C\int_{r_0+|t|}^{+\infty} \left(\left(\partial_r \tu(t)\right)^2+\left(\partial_t \tu(t)\right)^2\right)r^2\,dr+ C\left(\int_{r_0}^{+\infty}\left((\partial_r v_0)^2+v_1^2\right)\,dr+r_0u_0^2(r_0)\right)^5.
\end{multline*}
By finite speed of propagation, 
$$ \vec{u}(t,r)=\vecc{\tu}(t,r)\text{ and }\vec{u}_{\lin}(t,r)=\vecc{\tu}_{\lin}(t,r),\quad r\geq r_0+|t|,$$
and we obtain:
\begin{multline}
\label{nice_bound}
 \int_{r_0+|t|}^{+\infty} \left(\left(\partial_r u_{\lin}(t)\right)^2+\left(\partial_t u_{\lin}(t)\right)^2\right)r^2\,dr
\\
\leq C\int_{r_0+|t|}^{+\infty} \left(\left(\partial_r u(t)\right)^2+\left(\partial_t u(t)\right)^2\right)r^2\,dr+ C\left(\int_{r_0}^{+\infty}\left((\partial_r v_0)^2+v_1^2\right)\,dr+r_0u_0^2(r_0)\right)^5.
\end{multline}
Combining \eqref{bound_vL} and \eqref{nice_bound}, we see that the following holds for all $t\geq 0$ or for all $t\leq 0$:
\begin{multline}
\label{nice_bound2}
 \frac 12 \int_{r_0}^{+\infty} \left((\partial_r v_0)^2+v_1^2\right)\,dr \\
\leq C\int_{r_0+|t|}^{+\infty} \left(\left(\partial_r u(t)\right)^2+\left(\partial_t u(t)\right)^2\right)r^2\,dr+ C\left(\int_{r_0}^{+\infty}\left((\partial_r v_0)^2+v_1^2\right)\,dr+r_0u_0^2(r_0)\right)^5.
\end{multline}
Letting $t\to +\infty$ or $t\to -\infty$ in \eqref{nice_bound2}, we see that the first term of the right-hand side of \eqref{nice_bound2} goes to $0$ by our assumption on $u$. Since
$$ \int_{r_0}^{+\infty}\left((\partial_r v_0)^2+v_1^2\right)\,dr\leq \int_{r_0}^{+\infty}\left((\partial_r u_0)^2+u_1^2\right)r^2\,dr\leq \delta_0,$$
and $\delta_0$ is small, we can neglect the term $\int_{r_0}^{+\infty}\left((\partial_r v_0)^2+v_1^2\right)\,dr$ in the right-hand side of \eqref{nice_bound2}. Noting that $r_0^5 u^{10}_0(r_0)=\frac{v_0^{10}(r_0)}{r_0^5}$, we get \eqref{funct_eq}.
\end{proof}
\begin{lemma}
\label{L:limit}
The function $v_0(r)$ has a limit $\ell\in \RR$ as $r\to +\infty$. Furthermore, there exists $C>0$ such that 
\begin{equation}
\label{bound_v-l}
 \forall r\geq 1,\quad |v_0(r)-\ell|\leq \frac{C}{r^2}.
\end{equation} 
\end{lemma}
\begin{proof}
 We first claim that there exists $C>0$ such that for large $r$:
\begin{equation}
 \label{weak_bound_v}
|v_0(r)|\leq Cr^{1/10}.
\end{equation} 
Indeed by \eqref{bound_v0}, if $n\in \NN$, $|v_0(2^{n+1}r_0)|\leq (1+\sqrt{C_0})\delta_0^2|v_0(2^nr_0)|$. Hence by an elementary induction
$$ |v_0(2^nr_0)|\leq (1+\sqrt{C_0})^n\delta_0^{2n}|v_0(r_0)|.$$
Chosing a smaller $\delta_0$ if necessary, we can assume $(1+\sqrt{C_0})\delta_0^2\leq 2^{1/10}$, and thus
$$ |v_0(2^nr_0)|\leq 2^{\frac{n}{10}}|v_0(r_0)|,$$
which shows the inequality \eqref{weak_bound_v} for $r=2^nr_0$, $n\in \NN$. The general case for \eqref{weak_bound_v} follows from \eqref{bound_v0}.

We next prove that $v_0(r)$ has a limit as $r\to+\infty$. By \eqref{bound_v0}, we get, for $n\in \NN$,
$$ \left|v_0(2^nr_0)-v_0(2^{n+1}r_0)\right|\leq \sqrt{C_0}\frac{\left|v_0(2^nr_0)\right|^5}{\left(2^nr_0\right)^2}.$$
By \eqref{weak_bound_v}, there exists $C>0$ such that
$$ \left|v_0(2^nr_0)-v_0(2^{n+1}r_0)\right|\leq \frac{C}{(2^n)^{2-5/10}}=\frac{C}{2^{\frac{3}{2}n}}.$$
Using that $\sum \frac{1}{2^{\frac{3}{2}n}}$ converges, we get
$$ \sum_{n\geq 1} \left|v_0(2^nr_0)-v_0(2^{n+1}r_0)\right|<\infty,$$
which shows that there exists $\ell \in \RR$ such that 
$$ \lim_{n\to+\infty} v(2^nr_0)=\ell.$$
Using \eqref{bound_v0} and  \eqref{weak_bound_v}, we get
$$ \lim_{r\to+\infty} v_0(r)=\ell.$$

It remains to prove \eqref{bound_v-l}. Using that $v_0(r)$ converges as $r\to \infty$, we get that it is bounded for $r\geq r_0$, and thus the first inequality in \eqref{bound_v0} implies, for $r\geq r_0$ and $n\in \NN$,
$$\left|v_0(2^{n+1}r)-v_0(2^nr)\right|\leq \frac{C}{2^{2n}r^2}.$$
Summing up, we get
$$ \left|\ell - v(r)\right|=\left|\sum_{n\geq 0}(v(2^{n+1}r)-v(2^nr))\right|\leq \frac{C}{r^2}\sum_{n\geq 0}\frac{1}{4^n},$$
which concludes the proof of Lemma \ref{L:limit}.
\end{proof}

\begin{proof}[End of the proof of Lemma \ref{L:channel1}]
 Consider the limit $\ell$ of $v_0$ defined in Lemma \ref{L:limit}. We distinguish between two cases, depending on $\ell$.

\EMPH{The case $\ell=0$}

In this case we will show that $(v_0,v_1)$ is compactly supported. We fix a large $r$. By \eqref{bound_v0}, using that $\delta_0$ is small, 
$$\left|v_0(2^{n+1}r)\right|\geq \frac 34 \left|v_0(2^nr)\right|,\quad \forall n\in \NN.$$
By induction, we obtain $|v_0(2^nr)|\geq \left(\frac 34 \right)^{n} |v_0(r)|$. Since $\ell=0$, \eqref{bound_v-l} in Lemma \ref{L:limit} implies: $|v_0(2^nr)|\leq \frac{C}{4^n}$. Hence
$$ \forall n,\quad \frac{C}{4^n}\geq \left(\frac{3}{4}\right)^n |v_0(r)|.$$
Letting $n\to +\infty$, we get a contradiction unless $v_0(r)=0$. Since $r$ is any large positive number, we have shown that the support of $v_0$ is compact. By \eqref{funct_eq}, we get that the support of $v_1$ is also compact, concluding this case.

\EMPH{The case $\ell\neq 0$}
In this case we will show that there exists $\lambda>0$ and a sign $+$ or $-$ such that $\left(u_0\pm \frac{1}{\lambda^{1/2}}W\left(\frac{x}{\lambda}\right),u_1\right)$ is compactly supported. We note that for large $r$,
$$\left|\frac{1}{\lambda^{1/2}}W\left(\frac{r}{\lambda}\right)-\frac{\sqrt{3}\lambda^{1/2}}{r}\right|\leq \frac{C}{r^3}$$
Thus Lemma \ref{L:limit} implies the existence of a constant $C>0$ such that
$$ \left|\pm \frac{1}{\lambda^{1/2}}W\left(\frac{r}{\lambda^{1/2}}\right)-u_0(r)\right|\leq \frac{C}{r^3},\quad r\geq 1,$$
where $\lambda=\frac{\ell^2}{3}$ and the sign $\pm$ is the sign of $\ell$. Rescaling $u$, and replacing $u$ by $-u$ if $\ell<0$, we can assume:
\begin{equation}
 \label{u_close_W}
\left|u_0(r)-W(r)\right|\leq \frac{C}{r^3},\quad r\geq 1.
\end{equation} 

Let $h=u-W$, $H=rh$. We claim that for a large $R_0>0$ we have 
\begin{equation}
 \label{ineq_H}
\forall r_0>R_0,\quad \int_{r_0}^{+\infty} \left((\partial_rH_0)^2+H_1^2\right)\,dr\leq \frac{1}{16} \frac{H_0^2(r_0)}{r_0},
\end{equation} 
where $(H_0,H_1)=(H,\partial_t H)_{\restriction t=0}$. Assuming \eqref{ineq_H}, it is easy to conclude that $(H_0(r),H_1(r))=(0,0)$ for large $r$ exactly as in the case $\ell=0$. Indeed, \eqref{ineq_H} implies, for large $r$ and $n\in \NN$:
\begin{equation*}
 \left|H_0(2^{n+1}r)-H_0(2^nr)\right|\leq 2^{\frac{n}{2}}\sqrt{r}  \sqrt{\int_{2^n r}^{2^{n+1}r} (\partial_r H_0(s))^2\,ds}\leq 2^{\frac{n}{2}}\sqrt{r}\frac{|H_0(2^nr)|}{2^{\frac{n}{2}}\sqrt{r}}\times \frac{1}{4},
\end{equation*}
which implies $|H_0(2^{n+1}r)|\geq \frac{3}{4} |H_0(2^nr)|$. By an elementary induction, $|H_0(2^nr)|\geq \left(\frac{3}{4}\right)^n|H_0(r)|$. By \eqref{u_close_W}, we have $|H_0(2^nr)|\leq \frac{C}{4^nr^2}$. Letting $n\to +\infty$, we get a contradiction unless $H_0(r)=0$. Thus $H_0$ is compactly supported. By \eqref{ineq_H} again, we obtain that $H_1$ is compactly supported concluding the proof. 

It remains to show \eqref{ineq_H}. Consider the large positive number $R_0$ and the potential $V$ defined by Claim \ref{C:WV_OK}. Let $r_0>R_0$, and define 
$$ (g_0,g_1)=\Psi_{r_0}(h_0,h_1).$$
Let $g_{\lin}(t)=S(t)(g_0,g_1)$, and consider the solution $g$ of 
\begin{equation}
\label{CP_g3}
\left\{ 
\begin{gathered}
\partial_t^2 g -\Delta g=5V^4g+10V^3g^2+10V^2g^3+5Vg^4+g^5,\quad (t,x)\in \RR\times \RR^3\\
(g,\partial_t g)_{\restriction t=0}=(g_0,g_1).
\end{gathered}\right.
\end{equation}
By Lemma \ref{L:CPCarlos}, we get that $g$ is globally defined and satisfies
\begin{equation}
 \label{g_gL_bis}
\sup_{t\in \RR} \left\|\vec{g}(t)-\vec{g}_{\lin}(t)\right\|_{\hdot\times L^2} \leq \frac{1}{10} \|(g_0,g_1)\|_{\hdot\times L^2}.
\end{equation} 
By Lemma \ref{L:lin}, the following holds for all $t\geq 0$ or for all $t\leq 0$:
\begin{equation}
\label{channel_g3}
 \int_{r_0+|t|}^{+\infty} \left[(\partial_r(rg_{\lin}))^2+(\partial_t(r g_{\lin}))^2\right]\,dr\geq \frac 12 \int_{r_0}^{+\infty} \left[(\partial_r H_0)^2+H_1^2\right]\,dr.
\end{equation} 
Combining \eqref{g_gL_bis} and \eqref{channel_g3}, we get that for all $t\geq 0$ or for all $t\leq 0$,
\begin{multline}
\label{gros_calcul}
 \frac 12 \int_{r_0}^{+\infty} \left[(\partial_r H_0)^2+H_1^2\right]\,dr\leq \int_{r_0+|t|}^{+\infty}
\left[(\partial_r(rg_{\lin}))^2+(\partial_t(r g_{\lin}))^2\right]\,dr\\
\leq \int_{r_0+|t|}^{+\infty}
\left((\partial_r g_{\lin})^2+(\partial_t g_{\lin})^2\right)r^2\,dr\\
\leq \frac{1}{50}\int_{r_0}^{+\infty}\left((\partial_r g_0)^2+g_1^2\right)r^2dr+2\int_{r_0+|t|}^{+\infty} \left((\partial_rg)^2+(\partial_tg)^2\right)r^2\,dr.
\end{multline}
By finite speed of propagation (as in Step 2 of \S \ref{SS:cpct_supp}) we get 
$$ \vec{g}(t,r)=\vec{h}(t,r)\text{ for } r\geq r_0+|t|.$$
Letting $t\to +\infty$ or $t\to -\infty$ and using that
$$ \lim_{t\to \pm\infty} \int_{r_0+|t|}^{+\infty}(\partial_r W)^2r^2\,dr=0,$$
and our assumption on $u$, 
we obtain that the second term of the last line of \eqref{gros_calcul} goes to $0$. Hence \eqref{gros_calcul} implies
\begin{multline*}
 \frac 12 \int_{r_0}^{+\infty} \left[(\partial_r H_0)^2+H_1^2\right]\,dr\leq \frac{1}{50}\int_{r_0}^{+\infty}\left((\partial_r g_0)^2+g_1^2\right)r^2dr\\=\frac{1}{50}\left[\int_{r_0}^{+\infty} \left[(\partial_r H_0)^2+H_1^2\right]\,dr+\frac{1}{r_0}H_0^2(r_0)\right],
\end{multline*}
hence \eqref{ineq_H}.
\end{proof}
\section{Proof of the main result in the global case}
\label{S:global}
In this section we prove Theorem \ref{T:main} in the global case: we show that all global radial solutions of \eqref{CP}, can be expanded as in \eqref{expansion}. We start by recalling a few useful facts about the profile decomposition of Bahouri and G\'erard \cite{BaGe99}. In Subsection \ref{SS:boundedness}, we prove, using finite speed of propagation and convexity/monotonicity as in \cite{KeMe08}, that a global solution is bounded along a sequence of times going to infinity. In Subsection \ref{SS:free_wave}, we show that a global solution $u$ has a linear behaviour at finite distance from the boundary $\{|x|=|t|\}$ of the wave cone, thus constructing the free wave $v_{\lin}$ of the expansion \eqref{expansion}. The core of the proof is Subsection \ref{SS:sequences} where we use the channel of energy method and the results of Section \ref{S:channels} to prove that an expansion as \eqref{expansion} holds (after extraction of a subsequence) along any sequence of times going to infinity for which the solution is bounded. In Subsection \ref{SS:all_times} we conclude the proof, using continuity arguments to chose the signs $\iota_j$ and the scaling parameters $\lambda_j(t)$ independently of the choice of the sequence of times.
\subsection{Preliminaries on profile decomposition}
We gather in this subsection well known facts about the profile decomposition of Bahouri and G\'erard \cite{BaGe99}. 
\subsubsection{Definition}
Consider a sequence $\left\{(u_{0,n},u_{1,n})\right\}_n$ of radial functions in $\hdot\times L^2$, which is bounded in $\hdot\times L^2$. 
By \cite{BaGe99}, there exists a subsequence of $\left\{(u_{0,n},u_{1,n})\right\}_n$ (that we still denote by $\left\{(u_{0,n},u_{1,n})\right\}_n$) with the following properties.

There exist a sequence $(U^j_{\lin})_{j\geq 1}$ of radial solutions of the linear equation \eqref{lin_CP} with initial data $(U^j_0,U^j_1)\in \hdot\times L^2$, and, for $j\geq 1$, sequences $\{\lambda_{j,n}\}_n$, $\{t_{j,n}\}_n$ with $\lambda_{j,n}>0$, $t_{j,n}\in \RR$ satisfying the pseudo-orthogonality relation
\begin{equation}
\label{ortho_param}
j\neq k\Longrightarrow \lim_{n\to \infty} \frac{\lambda_{j,n}}{\lambda_{k,n}}+\frac{\lambda_{k,n}}{\lambda_{j,n}}+\frac{|t_{j,n}-t_{k,n}|}{\lambda_{j,n}}=+\infty.
\end{equation} 
such that, if
\begin{equation}
\label{decompo_profil}
\left\{\begin{aligned}
w_{0,n}^J(x)&:=u_{0,n}-\sum_{j=1}^J \frac{1}{\lambda_{j,n}^{\frac{1}{2}}}U_{\lin}^j\left(\frac{-t_{j,n}}{\lambda_{j,n}},\frac{x}{\lambda_{j,n}}\right),\\
w_{1,n}^J(x)&:=u_{1,n}-\sum_{j=1}^J \frac{1}{\lambda_{j,n}^{\frac{3}{2}}}\partial_t U_{\lin}^j\left(\frac{-t_{j,n}}{\lambda_{j,n}},\frac{x}{\lambda_{j,n}}\right),
\end{aligned}\right.
\end{equation}
then
\begin{equation}
\label{small_w}
\lim_{n\rightarrow+\infty}\limsup_{J\rightarrow+\infty} \left\|w_n^J\right\|_{L^8(\RR^4_{t,x})}=0,
\end{equation}
where
\begin{equation*}
 w_n^J(t)=S(t)(w_{0,n}^J,w_{1,n}^J).
\end{equation*} 
One says that $(u_{0,n},u_{1,n})_n$ admits a \emph{profile decomposition} with profiles $\lf\{U_{\lin}^j\rg\}_j$ and parameters $\lf\{\lambda_{j,n},t_{j,n}\rg\}_{j,n}$.

The profiles can be constructed as follows. Let $v_n(t)=S(t)(u_{0,n},u_{1,n})$. Then 
\begin{gather}
\label{weak_CV_Uj}
\left(\lambda_{j,n}^{1/2} v_n\left(t_{j,n},\lambda_{j,n}\cdot\right),\lambda_{j,n}^{3/2} \partial_tv_n\left(t_{j,n},\lambda_{j,n}\cdot\right)\right)\xrightharpoonup[n\to \infty]{}(U_0^j,U_1^j),\\
\label{weak_CV_wJ}
j\leq J\Longrightarrow 
\left(\lambda_{j,n}^{1/2} w_n^J\left(t_{j,n},\lambda_{j,n}\cdot\right),\lambda_{j,n}^{3/2} \partial_tw_n^J\left(t_{j,n},\lambda_{j,n}\cdot\right)\right)\xrightharpoonup[n\to \infty]{}0,
\end{gather}
weakly in $\hdot\times L^2$. In other words, the initial data $(U^j_0,U^j_1)$ of the profiles are exactly the weak limits, in $\hdot\times L^2$, of sequences $\left\{\lambda_n^{1/2}v_n(t_n,\lambda_n \cdot),\lambda_n^{3/2}\partial_tv_n(t_n,\lambda_n \cdot)\right\}$, where $\left\{\lambda_n\right\}_n$, $\left\{t_n\right\}_n$ are sequences in $(0,\infty)$ and $\RR$ respectively.

The following expansions hold for all $J\geq 1$:
\begin{gather}
\label{pythagore1a} 
\left\|u_{0,n}\rg\|_{\hdot}^2=\sum_{j=1}^J \left\|U^j_{\lin}\left(\frac{-t_{j,n}} {\lambda_{j,n}}\right)\rg\|_{\hdot}^2+\left\|w_{0,n}^J\rg\|_{\hdot}^2+o_n(1)\\
\label{pythagore1b} 
\lf\|u_{1,n}\right\|^2_{L^2}=\sum_{j=1}^J \lf\|\partial_t U^j_{\lin}\left(\frac{-t_{j,n}} {\lambda_{j,n}}\right)\right\|^2_{L^2}+\lf\|w_{1,n}^J\right\|^2_{L^2}+o_n(1)\\
\label{pythagore2}
E(v_{0,n},v_{1,n})=\sum_{j=1}^J E\left(U^j_{\lin}\left(-\frac{t_{j,n}}{\lambda_{j,n}}\right),\partial_t U^j_{\lin}\left(-\frac{t_{j,n}} {\lambda_{j,n}}\right)\right)+E\left(w_{0,n}^J, w_{1,n}^J\right)+o_n(1).
\end{gather}
We denote, for simplicity:
\begin{equation}
\label{simplicity} 
U^j_{\lin,n}(t,x)=\frac{1}{\lambda_{j,n}^{1/2}}U^j_{\lin}\left(\frac{t-t_{j,n}}{\lambda_{j,n}},\frac{x}{\lambda_{j,n}}\right),
\end{equation}

\subsubsection{Approximation by a sum of profiles}

Translating in time and rescaling $U^j_{\lin}(t,x)$, and extracting subsequences, we will always assume that one of the following two cases occurs
\begin{equation}
\label{choice_param}
\forall n,\;t_{j,n}=0\quad\text{or}\quad\lim_{n\to\infty} \frac{-t_{j,n}}{\lambda_{j,n}}=\pm\infty. 
\end{equation} 
As a consequence, using the local well-posedness of \eqref{CP} in the first case and the existence of wave operators for \eqref{CP} in the second case, one can construct a solution $U^j$ of \eqref{CP} such that $-t_{j,n}/\lambda_{j,n}$ is in the domain of $U^j$ for large $n$ and 
$$\lim_{n\to \infty} \|\vecc{U}^j(-t_{j,n}/\lambda_{j,n})-\vecc{U}^j_{\lin}(-t_{j,n}/\lambda_{j,n})\|_{\hdot\times L^2}=0.$$
The solution $U^j$ is called the \emph{nonlinear profile} associated to $U^j_{\lin}$, $\Big\{\lambda_{j,n},t_{j,n}\Big\}_n$. We will use the notation:
\begin{equation}
\label{simplicity2} 
U^j_n(t,x)=\frac{1}{\lambda_{j,n}^{1/2}}U^j\left(\frac{t-t_{j,n}}{\lambda_{j,n}},\frac{x}{\lambda_{j,n}}\right).
\end{equation}

We also recall the following approximation result, consequence of a long time perturbation argument. See the Main Theorem p. 135 in \cite{BaGe99} for the defocusing case, and a sketch of proof right after Proposition 2.8 in \cite{DuKeMe11a}.
\begin{prop}
\label{P:approx}
 Let $\{(u_{0,n},u_{1,n})\}_n$ be a bounded sequence in $\hdot\times L^2$ admitting a profile decomposition  with profiles $\{U^j_{\lin}\}$ and parameters $\{t_{j,n},\lambda_{j,n}\}$. Let $\theta_n\in [0,+\infty)$. Assume 
\begin{equation}
\label{bounded_strichartz}
\forall j\geq 1, \quad\forall n,\;\frac{\theta_n-t_{j,n}}{\lambda_{j,n}}<T_+(U^j)\text{ and } \limsup_{n\rightarrow +\infty} \left\|U^j\right\|_{L^8\big(\big(-\frac{t_{j,n}}{\lambda_{j,n}},\frac{\theta_n-t_{j,n}}{\lambda_{j,n}}\big)\times\RR^3\big)}<\infty.
\end{equation}
Let $u_n$ be the solution of \eqref{CP} with initial data $(u_{0,n},u_{1,n})$.
Then for large $n$, $u_n$ is defined on $[0,\theta_n)$,
\begin{equation}
\label{NL_bound}
\limsup_{n\rightarrow +\infty}\|u_n\|_{L^8\big((0,\theta_n)\times \RR^3\big)}<\infty,
\end{equation} 
and
\begin{equation}
\label{NL_profile} 
\forall t\in [0,\theta_n),\quad
u_n(t,x)=\sum_{j=1}^J U^j_n\left(t,x\right)+w^J_{n}(t,x)+r^J_n(t,x),
 \end{equation}
where  $w_n^J(t)=S(t)\left(w_{0,n}^J,w_{1,n}^J\right)$ and
\begin{equation}
\label{cond_rJn}
\lim_{J\rightarrow +\infty}\left[ \limsup_{n\rightarrow +\infty} \|r^J_n\|_{L^8\big((0,\theta_n)\times\RR^3\big)}+\sup_{t\in (0,\theta_n)} \left(\|\nabla r^J_n(t)\|_{L^2}+\|\partial_t r^J_n(t)\|_{L^2}\right)\right]=0.
\end{equation}
An analoguous statement holds if $\theta_n<0$.
\end{prop}
\subsubsection{An orthogonality property}
\begin{claim}
 \label{C:ortho}
 Let $\{(u_{0,n},u_{1,n})\}_n$, $\{U^j_{\lin}\}$, $\{t_{j,n},\lambda_{j,n}\}$ and $\theta_n\in \RR$ satisfy the assumptions of Proposition \ref{P:approx}. Consider sequences $\{\rho_n\}_n$, $\{\sigma_n\}_n$ such that for all $n$, $0\leq \rho_n<\sigma_n$ (the case $\sigma_n=+\infty$ is not excluded). Then:
\begin{gather}
\label{orthojk}
 j\neq k\Longrightarrow \lim_{n\to \infty} \int_{\rho_n\leq |x|\leq \sigma_n} \left(\nabla U^j_n(\theta_n,x)\cdot \nabla U_n^k(\theta_n,x)+\partial_t U^j_n(\theta_n,x)\cdot \partial_t U_n^k(\theta_n,x)\right)\,dx=0\\
\label{ortho_wnJ}
J\geq j\Longrightarrow \lim_{n\to \infty} \int_{\rho_n\leq |x|\leq \sigma_n} \left(\nabla U^j_n(\theta_n,x)\cdot \nabla w_n^j(\theta_n,x)+ \partial_t U^j_n(\theta_n,x)\cdot \partial_t w_n^J(\theta_n,x)\right)\,dx=0.
\end{gather}
\end{claim}
The proof of Claim \ref{C:ortho} is given in appendix \ref{A:ortho}.

\subsubsection{Localization of a profile} 
The following Lemma is an easy consequence of the strong Huygens principle. We refer to \cite{DuKeMe11a} for the proof.
\begin{lemma}
\label{L:lin_odd}
Let $U^j_{\lin,n}$ be defined by \eqref{simplicity}, and assume
$$\lim_{n\to \infty} \frac{-t_{j,n}}{\lambda_{j,n}}=\ell_j\in [-\infty,+\infty].$$
Then, if $\ell_j=\pm \infty$,
$$  \lim_{R\to \infty} \limsup_{n\to \infty} \int_{\big||x|-|t_{j,n}|\big|\geq R\lambda_{j,n}}  |\nabla U^j_{\lin,n}(0)|^2+\frac{1}{|x|^2}|U^j_{\lin,n}(0)|^2+\lf(\partial_t U^j_{\lin,n}(0)\rg)^2dx=0$$
and if $\ell_j\in \RR$,
$$  \lim_{R\to \infty} \limsup_{n\to \infty} \int_{\substack{\{|x|\geq R \lambda_{j,n}\}\\ \cup\{|x|\leq \frac{1}{R}\lambda_{j,n}\}}} |\nabla U^j_{\lin,n}(0)|^2+\frac{1}{|x|^2}|U^j_{\lin,n}(0)|^2+\lf(\partial_t U^j_{\lin,n}(0)\rg)^2dx=0.$$
\end{lemma}
\subsection{Boundedness along a subsequence}
\label{SS:boundedness}
\begin{prop}
 \label{P:boundedness}
Let $u$ be a solution of \eqref{CP} such that $T_+(u)=+\infty$. Then the energy of $u$ is nonnegative and
$$ \liminf_{t\to +\infty} \|\nabla u(t)\|^2_{L^2}+\|\partial_tu(t)\|^2_{L^2}\leq 3E(u_0,u_1).$$
In particular, there exists a sequence $t_n\to+\infty$ such that $\vec{u}(t_n)$ is bounded in $\hdot\times L^2$.
\end{prop}
\begin{remark}
Proposition \ref{P:boundedness} also holds in a nonradial context with the same proof.
\end{remark}
\begin{remark}
In \cite{KeMe08}, it was shown that if $\int |\nabla u_0|^2>\int |\nabla W|^2$ and $E(u_0,u_1)<E(W,0)$, then $T_+(u)$ is finite. In this case, the variational characterization of $W$ implies that for all $t$ in the domain of definition of $u$,
\begin{equation}
 \label{Frank_Carlos}
\int |\nabla u(t)|^2>\int |\nabla W|^2=3E(W,0),
\end{equation} 
which, together with the condition $E(u_0,u_1)<E(W,0)$, implies $\|\nabla u(t)\|^2_{L^2}+\|\partial_tu(t)\|^2_{L^2}\geq 3E(u_0,u_1)+\eps$ for some $\eps>0$ independent of $t$. Thus Proposition \ref{P:boundedness} implies the blow-up result of \cite{KeMe08}. The proof of Proposition \ref{P:boundedness} is almost the same as the one in \cite{KeMe08}, which uses an argument going back to H. Levine \cite{Levine74}. We sketch it for the sake of completeness.
\end{remark}
\begin{proof}
We argue by contradiction. Assume that the conclusion of the proposition does not hold. Then there exists $t_0>0$, $\eps_0>0$ such that
\begin{equation}
 \label{C3}
\forall t\geq t_0,\quad  \|\nabla u(t)\|^2_{L^2}+\|\partial_tu(t)\|^2_{L^2}\geq (3 +\eps_0)E(u_0,u_1)+\eps_0.
\end{equation}
Let 
$$y(t)=\int \varphi\lf(\frac{x}{t}\rg)|u(t,x)|^2\,dx.$$
We will show that there exists $\gamma>1$ such that for large $t$,
\begin{equation}
 \label{C9}
y'(t)>0,\text{ and }\gamma y'(t)^2\leq y(t)y''(t).
\end{equation}
This will gives a contradiction by standard ODE arguments. Indeed, \eqref{C9} implies that for large $t$,
$$ \frac{d}{dt} \log\left(\frac{y'(t)}{y^{\gamma}(t)}\right)\geq 0.$$
Thus there exists $c_0>0$ such that for large $t$,
$$ \frac{d}{dt}\left(\frac{1}{y^{\gamma-1}}\right)=(1-\gamma) \frac{y'(t)}{y^{\gamma}(t)}\leq -c_0,$$
which contradicts the fact that $y$ is nonnegative.

It remains to prove \eqref{C9}.
Combining finite speed of propagation, the small data Cauchy theory for \eqref{CP}, Hardy and Sobolev inequalities, we easily get that
\begin{equation}
\label{C2}
 \lim_{t\to +\infty} \int_{|x|\geq \frac{3}{2}t} \left(|\nabla u(t,x)|^2+\frac{1}{|x|^2}|u(t,x)|^2+ |u(t,x)|^6+(\partial_tu(t,x))^2\right)\,dx=0.
\end{equation}
Let $\varphi\in C^{\infty}(\RR^3)$ be a radial function such that $\varphi(r)=1$ if $r\leq 2$, $\varphi(r)=0$ if $r\geq 3$. Then
\begin{equation}
 \label{C4}
\forall t\geq 0,\quad y(t)\geq \int_{|x|\leq 2t} |u(t,x)|^2\,dx.
\end{equation}
Furthermore, 
\begin{equation}
 \label{C4'}
y'(t)=2\int u\partial_t u \varphi\left(\frac{x}{t}\right)\,dx-\frac{1}{t^2}\int u^2 x\cdot\nabla \varphi\lf(\frac{x}{t}\rg)\,dx,
\end{equation}
and thus by \eqref{C2},
\begin{equation}
 \label{C5}
|y'(t)|\leq 2\int_{|x|\leq 2t} |u|\,|\partial_tu|\,dx+o(t)\text{ as }t\to+\infty.
\end{equation}
Differentiating \eqref{C4'} and using equation \eqref{CP}, we get, in view of \eqref{C2},
\begin{align}
\notag
y''(t)&=2\int (\partial_t u)^2\,dx-2\int |\nabla u|^2\,dx+2\int u^6\,dx+o(1)\text{ as }t\to+\infty\\
 \label{C6}
y''(t)&=-12E(u_0,u_1)+8\int(\partial_tu)^2\,dx+4\int |\nabla u|^2\,dx+o(1)\text{ as }t\to+\infty.
\end{align}
By \eqref{C3}, there exists $t_1\geq t_0$ such that for some small $\eps_1>0$,
\begin{equation}
 \label{C7}
\forall t\geq t_1,\quad y''(t)\geq (4+\eps_1)\int (\partial_tu)^2\,dx+\eps_1.
\end{equation}
(Note that if $E(u_0,u_1)<0$, \eqref{C7} follows immediately from \eqref{C6} and we do not need \eqref{C3}. Of course this case was already treated in \cite{Levine74}, \cite{KeMe08}.)

In particular, $\liminf_{t\to+\infty} \frac{1}{t} y'(t)\geq \eps_1$, and \eqref{C5} implies that for large $t$,
\begin{multline}
 \label{C8}
0<y'(t)\leq \left(2+\frac{\eps_1}{100}\right)\int_{|x|\leq 2t} |u|\,|\partial_tu|\,dx\\
\leq \left(2+\frac{\eps_1}{100}\right)\left(\int_{|x|\leq 2t} |u|^2\,dx \right)^{1/2}\left(\int_{|x|\leq 2t} |\partial_tu|^2\,dx\right)^{1/2}.
\end{multline}
Combining \eqref{C4}, \eqref{C7} and \eqref{C8}, we get \eqref{C9}, concluding the proof of Proposition \ref{P:boundedness}.
\end{proof}

\subsection{Existence of the free wave}
\label{SS:free_wave}
\begin{lemma}
\label{L:vL}
 Let $u$ be a radial solution of \eqref{CP} such that $T_+(u)=+\infty$. Then there exists a radial solution $v_{\lin}$ of \eqref{lin_CP} such that
\begin{equation}
 \label{def_vL}
\forall A\in \RR,\quad \lim_{t\to +\infty}\int_{|x|\geq t-A} |\nabla (u-v_{\lin})(t,x)|^2+(\partial_t(u-v_{\lin})(t,x))^2\,dx=0.
\end{equation}
\end{lemma}
We first prove a preliminary result. Let $\{\varphi_{\delta}\}_{\delta}$ be a family of radial $C^{\infty}$ functions on $\RR^3$, defined for $\delta>0$ small and such that
\begin{equation}
 \label{phi_delta}
0\leq \varphi_{\delta}\leq 1,\quad |\nabla \varphi_{\delta}|\leq \frac{C}{\delta},\quad |x|\geq 1-\delta\Longrightarrow \varphi_{\delta}(x)=1,\text{ and } |x|\leq 1-2\delta\Longrightarrow \varphi_{\delta}(x)=0.
\end{equation}
\begin{lemma}
 \label{L:vL_prel}
Let $u$ be a solution of \eqref{CP} such that $T_+(u)=+\infty$, and $\eps$ be a small positive number. Then there exists $t_n\to+\infty$ and a small $\delta>0$ such that $\varphi_{\delta}\left(\frac{x}{t_n}\right)\vec{u}(t_n)$ has a profile decomposition with profiles $\lf\{U_{\lin}^j\rg\}_j$ and parameters $\lf\{\lambda_{j,n},t_{j,n}\rg\}_{j,n}$ such that
\begin{gather}
\label{Profile_vL1}
 \forall j\geq 2,\quad \lim_{n\to +\infty} \frac{-t_{j,n}}{\lambda_{j,n}}=+\infty,\\
\label{Profile_vL2}
t_{1,n}=0\quad \text{and}\quad \|(U_0^1,U_1^1)\|_{\hdot\times L^2}\leq \eps.
\end{gather}
\end{lemma}
\begin{proof}
The proof is very close to \cite[Proof of Lemma 3.8]{DuKeMe11P}. We recall it for the sake of completeness. We divide the proof in two steps.

\EMPH{Step 1}
In this step we show that there exist  $\delta'>0$ and a sequence $s_n\to +\infty$ such that
$\left\{\varphi_{\delta'}\left(\frac{x}{s_n}\right)\vec{u}(s_n)\right\}_n$ has a profile decomposition with profiles $\lf\{V_{\lin}^j\rg\}_j$ and parameters $\lf\{\mu_{j,n},s_{j,n}\rg\}_{j,n}$ satisfying
\begin{gather}
\label{Profile_vL1sn}
 \forall j\geq 2,\quad \lim_{n\to +\infty} \frac{-s_{j,n}}{\mu_{j,n}}\in\{\pm \infty\}\text{ and }\lim_{n\to \infty}\frac{-s_{j,n}}{s_n}\in [-1,2\delta'-1]\cup [1-2\delta',1],\\
\label{Profile_vL2sn}
s_{1,n}=0\quad\text{and}\quad \|(V_0^1,V_1^1)\|_{\hdot\times L^2}\leq \frac{\eps}{2}.
\end{gather}

First note that by finite speed of propagation and small data theory,
\begin{equation}
\label{FSP}
 \lim_{R\to+\infty} \limsup_{t\to+\infty} \int_{|x|\geq t+R} |\nabla u|^2+(\partial_tu)^2\,dx=0.
\end{equation}

By Proposition \ref{P:boundedness}, there exists a sequence $s_n\to+\infty$ such that $\|\vec{u}(s_n)\|_{\hdot\times L^2}$ is bounded.  After extraction of a subsequence in $n$, we know from \cite{BaGe99} that  $\left\{\vec{u}(s_n)\right\}_n$ has a profile decomposition with profiles $\{\tV^j_{\lin}\}_j$ and parameters  $\lf\{\mu_{j,n},s_{j,n}\rg\}_{j,n}$. By \eqref{FSP} and Lemma \ref{L:lin_odd}, for all $j$,
\begin{gather}
 \label{sjn_sn}
\lim_{n\to +\infty} \frac{|s_{j,n}|}{s_n}\leq 1\\
 \label{mujn_sn}
\lim_{n\to +\infty} \frac{\mu_{j,n}}{s_n}<\infty.
\end{gather}
(as usual, extracting subsequences, we can always assume that these limits exist).

If $\lim_{n\to \infty}\frac{\mu_{j,n}}{s_n}>0$ then we cannot have $\lim_{n\to+\infty}\frac{|s_{j,n}|}{\mu_{j,n}}=+\infty$ which would contradict \eqref{sjn_sn}. Thus we can assume $s_{j,n}=0$ for all $n$. Using the pseudo-orthogonality of the parameters, we deduce that there is at most one index $j$ such that $\lim_{n\to \infty}\frac{\mu_{j,n}}{s_n}>0$. We will assume that this index is $j=1$, and that $\mu_{1,n}=s_n$ for all $n$. By \eqref{FSP},
\begin{equation}
\label{support_tV}
 \supp (\tV^1_0,\tV^1_1)\subset \{|x|\leq 1\}.
\end{equation}
Then
$$\varphi_{\delta'}\left(\frac{x}{s_n}\right)\left(\frac{1}{s_n^{1/2}}\tV_0^1\left(\frac{x}{s_n}\right),\frac{1}{s_n^{3/2}}\tV_1^1\left(\frac{x}{s_n}\right)\right)=\left(\frac{1}{s_n^{1/2}}V_0^1\left(\frac{x}{s_n}\right),\frac{1}{s_n^{3/2}}V_1^1\left(\frac{x}{s_n}\right)\right),$$
where $\left(V_0^1,V^1_1\right)=\varphi_{\delta'}(x)\left(\tV_0^1,\tV_1^1\right)$. Using \eqref{phi_delta} and \eqref{support_tV}, one can easily show that \eqref{Profile_vL2sn} is satisfied for small $\delta'>0$.

Let $j\geq 2$, and distinguish two cases: 
\begin{itemize}
 \item If $s_{j,n}=0$ for all $n$, by quasi-orthogonality, $\lim_{n\to+\infty} \frac{\mu_{j,n}}{s_n}=0$, which shows by Lemma \ref{L:lin_odd},
\begin{equation}
 \label{nul_profile}
\lim_{n\to +\infty} \left\| \varphi_{\delta'}\left(\frac{x}{s_n}\right)\vecc{\tV}^j_{\lin,n}(0)\right\|_{\hdot\times L^2}=0
\end{equation}
 \item If $\lim_{n\to +\infty} \frac{s_{j,n}}{\mu_{j,n}}=\pm\infty$, then, denoting by
$$ \tau_j=\lim_{n\to\infty} \frac{-s_{j,n}}{s_n} \in [-1,+1],$$
we have, by Lemma \ref{L:lin_odd},
\begin{equation}
 \label{CV_profile}
\lim_{n\to\infty} \left\|\varphi_{\delta'}\left(\frac{x}{s_n}\right)\vecc{\tV}_{\lin,n}^j(0)-\varphi_{\delta'}\left(|\tau_j|\right)\vecc{\tV}_{\lin,n}^j(0)\right\|_{\hdot\times L^2}=0.
\end{equation}
In particular, if $|\tau_j|\leq 1-\delta'$, $\varphi_{\delta'}\left(\frac{x}{s_n}\right)\vecc{\tV}_{\lin,n}^j(0)$ goes to $0$ in $\hdot\times L^2$ as $n$ tends to infinity. 
\end{itemize}
 
We have:
\begin{equation}
\label{cut_profile}
 \varphi_{\delta'}\left(\frac{x}{s_n}\right)(u(s_n),\partial_tu(s_n))=\sum_{j=1}^J \varphi_{\delta'}\left(\frac{x}{s_n}\right)\vecc{\tV}^j_{\lin,n}(0)+\varphi_{\delta'}\left(\frac{x}{s_n}\right)\left(w_{0,n}^J,w_{1,n}^J\right),
\end{equation}
where
$$ \lim_{J\to\infty}\limsup_{n\to+\infty}\left\|S(t)\left(w_{0,n}^J,w_{1,n}^J\right)\right\|_{L^8(\RR^4)}=0.$$
By \cite[Claim 2.11]{DuKeMe11a},
\begin{equation}
\label{dispersionOK}
\lim_{J\to\infty}\limsup_{n\to+\infty}\left\|S(t)\left[ \varphi_{\delta'}\left(\frac{x}{s_n}\right) \left(w_{0,n}^J,w_{1,n}^J\right)\right]\right\|_{L^8(\RR^4)}=0.
\end{equation}
Combining \eqref{nul_profile}, \eqref{CV_profile}, \eqref{cut_profile} and \eqref{dispersionOK} we get that $\left\{\varphi_{\delta'}\left(\frac{x}{s_n}\right)\vec{u}(s_n)\right\}_n$ has a profile decomposition satisfying \eqref{Profile_vL1sn} and \eqref{Profile_vL2sn}, which concludes Step 1.

\EMPH{Step 2}

Let $u_n$ be the solution of \eqref{CP} with initial data $\varphi_{\delta'}\left(\frac{x}{s_n}\right)\vec{u}(s_n)$. Then by Proposition \ref{P:approx}, $u_n$ is defined on $[0,s_n/2]$ and
$$ \vec{u}_n(s_n/2)=\sum_{j=1}^J \vecc{V}_{n}^j(s_n/2)+\vecc{w}_n^J(s_n/2),$$
where 
$$V_n^j(t,x)=\frac{1}{\mu_{j,n}^{1/2}}V^j\left(\frac{t-s_{j,n}}{\mu_{j,n}},\frac{x}{\mu_{j,n}}\right),$$
and $V^j$ are the nonlinear profiles associated to the profiles $V^j_{\lin,n}$ defined in Step 1.

Let $t_n=\frac{3}{2}s_n$ and $\delta=\frac{\delta'}{3}$. By finite speed of propagation and the definition of $\varphi_{\delta'}$,
$$ |x|\geq  \left(\frac{3}{2}-\delta'\right)s_n=(1-2\delta) t_n\Longrightarrow \vec{u}_n(s_n/2,x)=\vec{u}(t_n,x).$$
Thus
$$\varphi_{\delta}\left(\frac{x}{t_n}\right)\vec{u}(t_n)=\varphi_{\delta}\left(\frac{x}{t_n}\right)\vec{u}_n(s_n/2)=\sum_{j=1}^J \varphi_{\delta}\left(\frac{x}{t_n}\right)\vecc{V}_{n}^j(s_n/2)+\varphi_{\delta}\left(\frac{x}{t_n}\right)\vecc{w}_n^J(s_n/2),$$
and the conclusion of the lemma follows from a similar analysis to the one at the end of Step 1.
\end{proof}
\begin{proof}[Proof of Lemma \ref{L:vL}]
\EMPH{Step 1}
In this step we show that for all $A\in \RR$, there exists a radial solution $v_{\lin}^A$ to the linear equation \eqref{lin_CP} such that
\begin{equation}
\label{vLA}
 \lim_{t\to+\infty} \int_{|x|\geq t-A} \left|\nabla(u-v_{\lin}^A)(t,x)\right|^2+\left(\partial_t(u-v_{\lin}^A)(t,x)\right)^2\,dx=0.
\end{equation}
Again, the proof is close to the one in \cite{DuKeMe11P}.
Consider the sequence $t_n$ given by Lemma \ref{L:vL_prel}, and let $u_n$ be the solution of \eqref{CP} with initial data $\varphi_{\delta}(x/t_n)\vec{u}(t_n,x)$ at $t=0$. It follows from \eqref{Profile_vL1}, \eqref{Profile_vL2} and Proposition \ref{P:approx} that for large $n$, $u_n$ is globally defined and scatters for positive times. We fix a large $n$ and let $\tilde{v}_{\lin,n}$ be the solution of the linear equation \eqref{lin_CP} such that
$$ \lim_{t\to+\infty}\|\vec{u}_n(t)-\vecc{\tilde{v}}_{\lin,n}(t)\|_{\hdot\times L^2}=0.$$
By finite speed of propagation, $\vec{u}(t_n+t,x)=\vec{u}_n(t,x)$ for $|x|\geq (1-\delta)t_n+t$, $t\geq 0$. Hence
$$ \lim_{t\to+\infty} \int_{|x|\geq  -\delta t_n+t} \left(\left|\nabla u(t,x)-\nabla \tilde{v}_{\lin}(t-t_n,x)\right|^2+\left|\partial_t u(t,x)-\partial_t\tilde{v}_{\lin}(t-t_n,x)\right|^2\right)\,dx=0.$$
Chosing $n$ large, so that $\delta t_n\geq A$, we get \eqref{vLA} with $v_{\lin}^A(t,x)= \tilde{v}_{\lin}(t-t_n,x)$, concluding this step.

\EMPH{Step 2: end of the proof}
Consider the sequence $\{t_n\}_n$ given by Proposition \ref{P:boundedness}, and assume,
after extraction of a subsequence in $n$, that $S(-t_n)\vec{u}(t_n)$ has a weak limit $(v_{0,\lin},v_{1,\lin})$, as $n$ tends to infinity,  in $\hdot\times L^2$. Furthermore, extracting again, we can assume that the sequence $\vec{u}(t_n)$ has a profile decomposition
\begin{equation}
 \label{C31}
\vec{u}(t_n)=\vec{v}_{\lin}(t_n)+\sum_{j=2}^J \vecc{U}_{\lin,n}^j(0) +(w_{0,n}^J,w_{1,n}^J).
\end{equation}
Note that in this decomposition, we have chosen the first profile as $U^1_{\lin}=v_{\lin}$, with parameters $\lambda_{1,n}=1$, $t_{1,n}=-t_n$, which is consistent with the definition of profiles as weak limits, see \eqref{weak_CV_Uj}.

Let $A\in \RR$ and $v_{\lin}^A$ be the linear solution given by Step 1. Then $\vec{u}(t_n)-\vecc{v}_{\lin}^A(t_n)$ has the following profile decomposition:
$$\vec{u}(t_n)-\vecc{v}_{\lin}^A(t_n)=\vec{v}_{\lin}(t_n)-\vecc{v}_{\lin}^A(t_n)+\sum_{j=2}^J \vecc{U}_{\lin,n}^j(0) +(w_{0,n}^J,w_{1,n}^J),$$
where the first profile is $\tU^1_{\lin}=v_{\lin}-v_{\lin}^A$, and the corresponding parameters are again $\lambda_{1,n}=1$ and $t_{1,n}=-t_n$. By Claim \ref{C:ortho}, we get that \eqref{vLA} implies
\begin{equation*}
\lim_{n\to\infty} \int_{|x|\geq t_n-A}\left( |\nabla (v_{\lin}^A-v_{\lin})(t_n,x)|^2+(\partial_t (v_{\lin}^A-v_{\lin})(t_n,x))^2\right)\,dx=0.
\end{equation*}
Using that $v_{\lin}^A-v_{\lin}$ is a solution to the linear wave equation, the decay of the free energy of $v_{\lin}^A-v_{\lin}$ outside the lightcone $\big\{|x|\geq t-A\big\}$ implies:
\begin{equation*}
\lim_{t\to+\infty} \int_{|x|\geq t-A}\left( |\nabla (v_{\lin}^A-v_{\lin})(t,x)|^2+(\partial_t (v_{\lin}^A-v_{\lin})(t,x))^2\right)\,dx=0,
\end{equation*}
which, together with \eqref{vLA}, yields  \eqref{def_vL}.
\end{proof}

\subsection{Analysis along a sequence of times}
\label{SS:sequences}
In this subsection, we show:
\begin{prop}
\label{P:sequence}
 Let $t_n\to +\infty$ be such that $\left\{\vec{u}(t_n)\right\}_n$ is bounded in $\hdot\times L^2$, and $v_{\lin}$ be the linear solution given by Lemma \ref{L:vL}. Then, after extraction of a subsequence in $n$, there exist $J\geq 0$, $\iota_1,\ldots,\iota_J\in\{\pm 1\}$ and sequences $\{\lambda_{j,n}\}_n$ with $0<\lambda_{1,n}\ll \ldots \ll \lambda_{J,n}\ll t_n$ such that
\begin{equation}
\label{dev_sequence}
\vec{u}(t_n)-\vec{v}_{\lin}(t_n)-\sum_{j=1}^J \left(\frac{\iota_j}{\lambda_{j,n}^{1/2}}W\left(\frac{x}{\lambda_{j,n}},0\right),0\right) \underset{n\to +\infty}{\longrightarrow} 0
\end{equation} 
in $\hdot\times L^2$. 
\end{prop}
Let us emphasize the difference between Proposition \ref{P:sequence} and Theorem 4 of \cite{DuKeMe11P}. Theorem 4 of  \cite{DuKeMe11P} states that if $\vec{u}$ is bounded in $\hdot\times L^2$, there exists a sequence $t_n\to +\infty$ such that \eqref{dev_sequence} holds, whereas in Proposition \ref{P:sequence}, the sequence $t_n\to+\infty$ can be chosen as a subsequence of any sequence $\{t_n'\}_n$ such that $\vec{u}(t_n')$ is bounded. This apparently small difference allows us to prove that the expansion \eqref{expansion} holds for all large time, and not only along a sequence of times as in \cite{DuKeMe11P}. 

Let us quickly explain the proof of Proposition \ref{P:sequence}. Arguing by contradiction, we expand $\vec{u}(t_n)$ into profiles and assume for example that one of the nonzero profiles is not equal to $\pm W$. Using the results of Section \ref{S:channels}, we show that this profile will send an energy channel into the future (which contradicts Lemma \ref{L:vL}) or into the past (giving an initial data with infinite energy, a contradiction). This \emph{channel of energy} method was already used in our previous articles \cite{DuKeMe11a}, \cite{DuKeMe10P} and \cite{DuKeMe11P}. However, in these articles, we could only show that small solutions of \eqref{CP} (and also, in \cite{DuKeMe11P}, compactly supported solutions) have an appropriate energy channel property, whereas Section \ref{S:channels} shows that this property holds in some sense for any nonstationary radial solution of \eqref{CP}. 

Before proving Proposition \ref{P:sequence}, we will need two technical lemmas. Lemma \ref{L:channel_profile} gives a ``profile'' version of the results of Section \ref{S:channels}. Lemma \ref{L:channel} makes explicit the energy channel argument.
\begin{lemma}
\label{L:channel_profile}
 Consider a non-zero profile 
$$ U_{\lin,n}^j(t,x)=\frac{1}{\lambda_{j,n}^{1/2}} U^j_{\lin}\left(\frac{t-t_{j,n}}{\lambda_{j,n}},\frac{x}{\lambda_{j,n}}\right), \quad U_{\lin}^j(t)=S(t)(U^j_0,U^j_1).$$
and assume that
\begin{equation}
 \label{infinite_lim}
 \lim_{n\to \infty}\frac{-t_{j,n}}{\lambda_{j,n}}\in \{\pm\infty\}
\end{equation}
or that $t_{j,n}=0$ for all $n$ and that one of the following holds:
\begin{enumerate}
\item for all $\mu>0$, for both signs $+$ or $-$, $\left(U_0^j\pm \frac{1}{\mu^{1/2}}W\left(\frac{\cdot}{\mu}\right),U^j_1\right)$ is not compactly supported, or
\item there exists a sign $+$ or $-$ such that $\left(U_0^j\pm W,U^j_1\right)$ is compactly supported and 
$$\rho\left(U_0^j\pm W,U^j_1\right)>R_0,$$
where $\rho$ is defined in \eqref{def_rho} and the constant $R_0>0$ is given by Proposition \ref{P:W+compct}.
\end{enumerate}
Then there exists a solution $\tU_{\lin}^j$ of the linear wave equation, and a sequence $\left\{\rho_{j,n}\right\}_n$ of positive numbers such that the nonlinear profile $\tU^j$ associated to $\tU_{\lin}^j$, $\{t_{j,n},\lambda_{j,n}\}_n$ is globally defined and scatters in both time directions, 
\begin{equation}
 \label{B18'}
\forall n, \quad |x|> \rho_{j,n}\Longrightarrow \vecc{\tU}^j_{\lin,n}(0,x)=\vecc{U}^j_{\lin,n}(0,x)\\
\end{equation}
and there exists $\eta_j>0$ such that the following holds for all $t\geq 0$ or for all $t\leq 0$
\begin{equation}
 \label{B18''}
\forall n,\quad \int_{|x|>\rho_{j,n}+|t|} \left|\nabla \tU_n^j(t,x)\right|^2+\left|\partial_t \tU_n^j(t,x)\right|^2\,dx \geq \eta_j.
\end{equation}
\end{lemma}
We postpone the proof of Lemma \ref{L:channel_profile} to Appendix \ref{A:channel_profile}.

\begin{lemma}
 \label{L:channel}
There exists no sequence $\{t_n\}_n\to +\infty$ with the following properties.\\ 
There exists a sequence of functions $\{(u_{0,n},u_{1,n})\}_n$, bounded in $\hdot\times L^2$, and a sequence $\{\rho_n\}_n$ of nonnegative numbers such that
\begin{equation}
 \label{B9}
|x|\geq \rho_n\Longrightarrow (u(t_n,x),\partial_tu(t_n,x))=(u_{0,n}(x),u_{1,n}(x)),
\end{equation} 
and there exists $J_0\in \NN$, $\iota_1,\ldots,\iota_{J}\in \{\pm 1\}$ such that $(u_{0,n},u_{1,n})$ has a profile decomposition of the following form:
\begin{equation}
\label{B10}
(u_{0,n},u_{1,n})=\vec{v}_{\lin}(t_n)+\sum_{j=1}^{J_0} \left(\frac{\iota_j}{\lambda_{j,n}^{1/2}}W\left(\frac{x}{\lambda_{j,n}}\right),0\right)+\sum_{j=J_0+1}^J \vecc{U}_{\lin,n}^j(0)+(w_{0,n}^J,w_{1,n}^J), 
\end{equation} 
where for all $j\geq J_0+1$, the nonlinear profile $U^j$ is globally defined and scatters in both time directions. Furthermore, there exists $\eps_0>0$ such that one of the following holds:
\begin{enumerate}
 \item \label{I:B12} there exists $j_0\geq J_0+1$ such that for all $t\geq 0$ or for all $t\leq 0$:
\begin{equation}
\label{B12}
\forall n \quad \int_{|x|\geq \rho_n+|t|} |\nabla U_n^{j_0}(t,x)|^2+(\partial_t U_n^{j_0}(t,x))^2\,dx \geq \eps_0
\end{equation} 
or
\item \label{I:B12'} for at least one sign $+$ or $-$,
\begin{equation}
 \label{B12'} \lim_{J\to+\infty}\,\liminf_{n\to+\infty}\,  \inf_{\pm t\geq 0} \int_{|x|\geq \rho_n+|t|} |\nabla w_n^{J}(t,x)|^2+(\partial_t w_n^{J}(t,x))^2\,dx \geq \eps_0.
\end{equation} 
\end{enumerate}
\end{lemma}
\begin{proof}
We first note that for any $j\in \{1,\ldots, J_0\}$, 
\begin{equation}
\label{lambdajn_negligible}	
 \lim_{n\to +\infty} \frac{\lambda_{j,n}}{t_{n}}=0.
\end{equation} 
This follows from \eqref{FSP}, the fact that $W$ is not compactly supported, and the formula
\begin{equation*}
 \left(\lambda_{j,n}^{1/2}u(t_n,\lambda_{j,n}\cdot),\lambda_{j,n}^{3/2}\partial_t u(t_n,\lambda_{j,n}\cdot)\right)\xrightharpoonup[n\to+\infty]{}  (\iota_j W,0)
\end{equation*} 

We denote by $v$ the solution of \eqref{CP} such that 
\begin{equation}
\label{v_vL}
\lim_{t\to +\infty} \|\vec{v}(t)-\vec{v}_{\lin}(t)\|_{\hdot\times L^2}=0. 
\end{equation} 
Translating $u$ in time if necessary, we will assume that $v$ is defined on $[0,+\infty)$.

We will prove the result by induction on $J_0$. 

\EMPH{Case $J_0=0$}

Let $u_n$ be the solution of \eqref{CP} with data $(u_{0,n},u_{1,n})$. By Proposition \ref{P:approx}, $(u_{0,n},u_{1,n})$ is defined on $[-t_n,+\infty)$ for large $n$ and 
\begin{equation}
 \label{B13} \vec{u}_n(t,x)=\vec{v}(t_n+t,x)+\sum_{j=1}^J\vecc{U}_n^j(t,x)+\vecc{w}_n^J(t,x)+\vecc{r}_n^J(t,x),
\end{equation} 
where
\begin{equation}
 \label{B14}
\lim_{J\to +\infty} \limsup_{n\to +\infty} \sup_{t\in [-t_n,+\infty)} \left\|\vecc{r}_n^J\right\|_{\hdot\times L^2}=0.
\end{equation} 
First assume that \eqref{B12} holds for all $t\geq 0$ or that \eqref{B12'} holds with a $+$ sign. Then by \eqref{v_vL}, \eqref{B13}, \eqref{B14},  \eqref{B12} (or \eqref{B12'}) and the orthogonality Claim \ref{C:ortho}, the following holds for all large $n$ and all $t\geq 0$:
\begin{equation}
\label{B15}
\int_{|x|\geq \rho_n+t} \left(\left|\nabla u_n(t,x)-\nabla v_{\lin}(t_n+t,x)\right|^2+(\partial_t u_n(t,x)-\partial_t v_{\lin}(t_n+t,x))^2\right)\,dx\geq \frac{\eps_0}{2}.
 \end{equation} 
By finite speed of propagation and \eqref{B9}, we deduce that for large $n$,
\begin{equation}
\label{B16}
\int_{|x|\geq \rho_n+t} \left(\left|\nabla u(t_n+t,x)-\nabla v_{\lin}(t_n+t,x)\right|^2+(\partial_t u(t_n+t,x)-\partial_t v_{\lin}(t_n+t,x))^2\right)\,dx\geq \frac{\eps_0}{2},
 \end{equation} 
and thus, 
\begin{equation}
\liminf_{t\to+\infty} \int_{|x|\geq \rho_n-t_n+t} \left(\left|\nabla u(t,x)-\nabla v_{\lin}(t,x)\right|^2+(\partial_t u(t,x)-\partial_t v_{\lin}(t,x))^2\right)\,dx>0,
\end{equation}
contradicting \eqref{def_vL}.

Next, we assume that \eqref{B12} holds for all $t\leq 0$, or that \eqref{B12'} holds with a $-$ sign. By \eqref{B13} at $t=-t_n$, \eqref{B14}, \eqref{B12} (or \eqref{B12'}) and the orthogonality Claim \ref{C:ortho}, we get that for large $n$ 
$$\int_{|x|\geq \rho_n+t_n} |\nabla u_n(-t_n,x)-\nabla v(0,x)|^2+(\partial_tu(-t_n,x)-\partial_t v(0,x))^2\,dx\geq \frac{\eps_0}{2}.$$
Using again \eqref{B9} and finite speed of propagation, we deduce that for large $n$,
$$\int_{|x|\geq \rho_n+t_n} \left(|\nabla u_0(x)-\nabla v(0,x)|^2+(u_1(x)-\partial_t v(0,x))^2\right)\,dx\geq \frac{\eps_0}{2}.$$
Letting $n\to+\infty$, we get again a contradiction.

\EMPH{Inductive step}
This part of the proof is close to \cite[Proof of Lemma 4.5]{DuKeMe11P}.
Fix $J_1\geq 0$, and assume that the lemma holds when $J_0\leq J_1$. Consider a sequence $t_n\to+\infty$ satisfying the assumptions of the lemma with $J_0=J_1+1$. We assume to fix ideas that \eqref{B12} or \eqref{B12'} holds for all $t\geq 0$. The proof is the same in the other case. Reordering the profiles (and extracting a subsequence if necessary), we may assume
$$ \lambda_{1,n}\ll \lambda_{2,n}\ll \ldots \ll\lambda_{J_0,n}\ll t_n.$$
Let $T>0$ be a large time. Using that $W$ is globally defined, we get by Proposition \ref{P:approx} and the fact that the nonlinear profiles $U^j$ scatter for $j\geq J_0+1$,
\begin{multline}
\label{B17}
 \vec{u}_n(\lambda_{1,n}T)\\
=\vec{v}_{\lin}(t_n+\lambda_{1,n}T)+\sum_{j=1}^{J_0} \left(\frac{\iota_j}{\lambda_{j,n}^{1/2}}W\left(\frac{x}{\lambda_{j,n}}\right),0\right)+\sum_{j=J_0+1}^J \vecc{U}_n^j(\lambda_{1,n}T)+\vecc{w}_n^{J}(\lambda_{1,n}T) +\vecc{r_n}^{J}(\lambda_{1,n}T),
\end{multline} 
where $\lim_{J\to\infty}\limsup_{n\to \infty} \left\|\vecc{r}^J_n\left(\lambda_{1,n}T\right)\right\|_{\hdot\times L^2}=0$. 

Let $\left(\tU_0^1,0\right)=\Psi_T(W,0)$, where $\Psi_T$ is defined in the beginning of Subsection \ref{SS:preliminaries}. Chosing $T$ large, we can assume that the solution $\tU$ with initial data $(\tU_0^1,0)$ is globally defined and scatters in both time directions. Let
\begin{multline}
\label{B18}
(\tu_{0,n},\tu_{1,n})=\vec{v}_{\lin}(t_n+\lambda_{1,n}T)+\left(\frac{\iota_1}{\lambda_{1,n}^{1/2}}\tU_0^1\left(\frac{x}{\lambda_{1,n}} \right),0\right)
\\+\sum_{j=2}^{J_0} \left(\frac{\iota_j}{\lambda_{j,n}^{1/2}}W\left(\frac{x}{\lambda_{j,n}},0\right)\right)+\sum_{j=J_0+1}^{J} \vecc{U}_n^j(\lambda_{1,n}T)+\vecc{w}_n^J(\lambda_{1,n}T)+\vecc{r}_n^J(\lambda_{1,n}T). 
\end{multline}
We check that the sequences $\tilde{t}_n=t_n+\lambda_{1,n}T$, $\trho_n=\rho_n+\lambda_{1,n}T$ and $\{(\tu_{0,n},\tu_{1,n})\}_n$ satisfy the assumptions of Lemma \ref{L:channel} with $J_0-1$ instead of $J_0$. 

By finite speed of propagation,
$$ (\tu_{0,n},\tu_{1,n})=\vecc{\tu}_n(\lambda_{1,n}T,x)=\vecc{\tu}\left(\ttt_n,x\right)\quad\text{for } |x|\geq \rho_n+\lambda_{1,n}T=\trho_n.$$
The expansion \eqref{B18} yields a profile decomposition of $(\tu_{0,n},\tu_{1,n})$:
\begin{multline*}
 (\tu_{0,n},\tu_{1,n})=\vec{v}_{\lin}(\ttt_n)+\sum_{j=2}^{J_0} \left(\frac{\iota_j}{\lambda_{j,n}^{1/2}}W\left(\frac{x}{\lambda_{j,n}}\right),0\right)+\left(\frac{\iota_1}{\lambda_{1,n}^{1/2}}\tU_0^1\left(\frac{x}{\lambda_{1,n}} \right),0\right)\\
+\sum_{j=J_0+1}^J \left(\frac{1}{\lambda_{j,n}^{1/2}}U^j\left(\frac{-\ttt_{j,n}}{\lambda_{j,n}},\frac{x}{\lambda_{j,n}}\right),\frac{1}{\lambda_{j,n}^{3/2}}\partial_t U^j\left(\frac{-\ttt_{j,n}}{\lambda_{j,n}},\frac{x}{\lambda_{j,n}}\right)\right)+\left(\tw_{0,n}^{J},\tw_{1,n}^J\right),
\end{multline*} 
where $\ttt_{j,n}=-\lambda_{1,n}T+t_{j,n}$ (note that this preserves the pseudo-orthogonality of the sequence of parameters $\{\lambda_{j,n}\}_n$, $\{\ttt_{j,n}\}_n$) and
$$(\tw_{0,n}^J,\tw_{1,n}^J)=\vec{w}_n(\lambda_{1,n}T)+\vec{r}_n^J(\lambda_{1,n}T).$$
By the small data theory, the solution $\tU$ of \eqref{CP} with initial data $(\iota_1\tU_0^1,0)$ is globally defined and scatters in both time directions.
Finally, if \eqref{B12} holds then
$$ \forall t\geq 0,\quad \int_{|x|\geq \rho_n+\lambda_{1,n}T+t}\left|\nabla U_n^{j_0}(\lambda_{1,n}T+t,x)\right|^2+\left|\partial_t U_n^{j_0}(\lambda_{1,n}T+t,x)\right|^2\,dx\geq \eps_0.$$
Letting $\tU_n^{j}=\frac{1}{\lambda_{j,n}^{1/2}}U^j\left(\frac{t-\ttt_{j,n}}{\lambda_{j,n}},\frac{x}{\lambda_{j,n}}\right)=U_n^j(t+\lambda_{1,n}t,x)$, we obtain
$$ \forall t\geq 0,\quad \int_{|x|\geq \trho_n+t}\left|\nabla \tU_n^{j_0}(t,x)\right|^2+\left|\partial_t \tU_n^{j_0}(t,x)\right|^2\,dx\geq \eps_0.$$
Similarly, if \eqref{B12'} holds we get:
$$ \lim_{J\to +\infty}\,\liminf_{n\to \infty}\,\inf_{t\geq 0}\int_{|x|\geq \trho_{n}+t} |\nabla \tw_n^{J}(t,x)|^2+(\partial_t \tw_n^{J}(t,x))^2\,dx\geq \eps_0.$$
We are reduced to $J_0-1$ profiles $W$, which closes the induction argument.
\end{proof}

We are now in position to prove Proposition \ref{P:sequence}. We argue by contradiction. If the conclusion of the proposition does not hold, the exists a subsequence of $\{t_n\}_n$ (still denoted by $\{t_n\}_n$) such that $\vec{u}(t_n)$ has a profile decomposition of the following form:
\begin{equation}
 \label{D1}
\vec{u}(t_n)=\vec{v}_{\lin}(t_n)+\sum_{j=1}^{J_0} \left(\frac{\iota_j}{\lambda_{j,n}^{1/2}}W\left(\frac{x}{\lambda_{j,n}}\right),0\right)+\sum_{j=J_0+1}^{J}\vecc{U}_{\lin,n}^j(0)+\left(w_{0,n}^J,w_{1,n}^J\right),
\end{equation} 
where $J_0\geq 0$, $\iota_j\in\{\pm 1\}$ and, for $j\geq J_0+1$, one of the following holds
\begin{equation}
 \label{D2}
\lim_{n\to\infty} \frac{-t_{j,n}}{\lambda_{j,n}}\in \{\pm \infty\}
\end{equation} 
or
\begin{equation}
 \label{D3}
\forall j\geq J_0+1,\; t_{j,n}=0\quad\text{and}\quad \forall \lambda>0,\; \left(U_{0,\lin}^j,U_{1,\lin}^j\right)\neq \left(\pm\frac{1}{\lambda^{1/2}}W\left(\frac{x}{\lambda}\right),0\right),
\end{equation}
Furthermore, one of the following holds:
\begin{equation}
 \label{D4}
U^{J_0+1}_{\lin}\neq 0
\end{equation} 
or
\begin{equation}
 \label{D5}
\forall j\geq J_0+1,\; U^{j}=0\quad\text{and}\quad \liminf_{n\to\infty} \left\|(w_{0,n}^{J_0},w_{1,n}^{J_0})\right\|_{\hdot\times L^2}>0.
\end{equation} 
We split the proof in various cases. In each case, using in particular Lemma \ref{L:channel_profile}, we reduce to the situation where $\vec{u}(t_n)$ coincides for $|x|>\rho_n$ (for some nonnegative parameter $\rho_n$), with a sum of rescaled $W$ and of globally defined profiles creating energy channels in the cone $\{|x|> \rho_n+|t|\}$. Lemma \ref{L:channel} will then yield a contradiction. This argument can be performed directly along the sequence $\{t_n\}_n$ (see cases 1, 2a and 2b below) unless the profile $U^j$, $j\geq J_0+1$ which is ``further'' from the origin $x=0$, is of the form $\left(W+h_0^j,h_1^j\right)$, where $\rho(h_0^j,h_1^j)$ is small. In this case, we will use case \eqref{I:small_supp} in Proposition \ref{P:W+compct}, finite speed of propagation and Proposition \ref{P:approx} to get the same situation along another sequence $\{\ttt_n\}_n$ (see Case 2c).  

\EMPH{Case 1} Assume that \eqref{D5} holds. As a consequence, $w_n^J=S(t)\left(w_{0,n}^J,w_{1,n}^J\right)$ is independent of $J\geq J_0+1$ and we will simply denote it by $w_n$. 
There exists $N_0>0$ and a small $\eps_0>0$ such that for $n\geq N_0$, $\|(w_{0,n},w_{1,n})\|_{\hdot\times L^2}\geq \eps_0$. Using that (letting $R\to 0$ in \eqref{IPP1}):
$$\int_{0}^{+\infty}\left[\left(\partial_r(rw_{0,n}(r))\right)^2+(rw_{1,n}(r))^2\right]\,dr=\int_{\RR^3} \left(|\nabla w_{0,n}|^2+(w_{1,n})^2\right)\,dx,$$
we get by Lemma \ref{L:lin} that the following holds for all $t\geq 0$ or for all $t\leq 0$:
\begin{multline*}
\forall n\geq N_0,\quad
\int_{|x|\geq |t|}\left(|\nabla w_n(t,x)|^2+(\partial_t w_n(t,x))^2\right)\,dx\\
\geq \int_{|t|}^{+\infty} \left[(\partial_r(rw_n(t,r)))^2+(\partial_t(rw_n(t,r)))^2\right]\,dr\geq \frac{\eps_0}{2}.
\end{multline*}
We are thus exactly in the setting of Lemma \ref{L:channel}, with $(u_{0,n},u_{1,n})=(u(t_n),\partial_tu(t_n))$, and $\rho_n=0$, which gives a contradiction.

\EMPH{Case 2} Assume that \eqref{D4} holds, and chose a small parameter $\eps>0$ such that 
\begin{equation}
 \label{D6}
\eps\leq \left\|(U_{\lin}^{J_0+1}(0),\partial_tU_{\lin}^{J_0+1}(0))\right\|_{\hdot\times L^2},
\end{equation} 
and that any solution $v$ of \eqref{CP} with initial data $(v_0,v_1)$ satisfying $\|(v_0,v_1)\|_{\hdot\times L^2}\leq 10\eps$ is globally defined and scatters.

Reordering the profiles again, we may assume that there exist $J_1,J_2$, with $J_0\leq J_1\leq J_2$ such that
\begin{align}
 \label{D7}
J_0+1\leq j\leq J_2&\Longrightarrow \left\|\left(U_{\lin}^j(0),\partial_t U_{\lin}^j(0)\right)\right\|_{\hdot\times L^2} \geq\eps\\
\label{D8}
J_2+1\leq j&\Longrightarrow \left\|\left(U_{\lin}^j(0),\partial_t U_{\lin}^j(0)\right)\right\|_{\hdot\times L^2} <\eps
\end{align}
and
\begin{itemize}
 \item if $J_0+1\leq j\leq J_1$, $t_{j,n}=0$ for all $n$ and $\left(U_0^j,U_1^j\right)=\iota_j(W,0)+(h_0^j,h_1^j)$, where $\iota_j\in \{\pm 1\}$ and $(h_0^j,h_1^j)\in \hdot \times L^2$ is nonzero and compactly supported;
 \item if $J_1+1\leq j\leq J_2$, then $\lim_{n\to +\infty} t_{j,n}/\lambda_{j,n}=\pm \infty$, or $t_{j,n}=0$ for all $n$ and for all $\lambda>0$, $\left(U_0^j(x),U_1^j(x)\right)\pm \left(\frac{1}{\lambda^{1/2}}W\left(\frac{x}{\lambda}\right),0\right)$ is not compactly supported;
\end{itemize}
Note that by \eqref{D6}, we must have $J_0+1\leq J_2$. 

In order to distinguish between the three remaining cases, we will need to define new sequences of parameters $\{\rho_{j,n}\}_n$ for $J_0+1\leq j\leq J_2$.

If $J_0+1\leq j\leq J_1$, we will denote by $\rho_{j,n}=\rho(h_0^j,h_1^j)\lambda_{j,n}$, where $\rho(\cdot)$ is defined in \eqref{def_rho}. Reordering the profiles and extracting subsequences, we will assume
\begin{equation}
\label{D9}
\lambda_{J_0+1,n}\ll \ldots \ll \lambda_{J_1,n}.
\end{equation} 
Equivalently
\begin{equation}
\label{D10}
\rho_{J_0+1,n}\ll \ldots \ll \rho_{J_1,n}.
\end{equation} 
By Lemma \ref{L:channel_profile}, if $J_1+1\leq j\leq J_2$, there exists $\left(\tU_{0,\lin}^j, \tU_{1,\lin}^j\right)$
such that the nonlinear profile $\tU^J$ associated to $\tU^j_{\lin}$, $\left\{t_{j,n}\right\}_n$, $\left\{\lambda_{j,n}\right\}_n$ is globally defined, scatters, and satisfies \eqref{B18'}, \eqref{B18''} for some $\rho_{j,n}>0$. Reordering the profiles and extracting subsequences, we will assume:
\begin{equation}
\label{D13}
\rho_{J_1+1,n}\leq \ldots \leq \rho_{J_2,n}
\end{equation}
If $J_0<J_1<J_2$ we can assume, after extraction of a subsequence in $n$ that the following limit exists
$$ \ell=\lim_{n\to\infty} \frac{\rho_{J_2,n}}{\rho_{J_1,n}}\in [0,+\infty].$$
We will make the following conventions:
if $J_1=J_0$ (i.e. $\{\rho_{J_1,n}\}_n$ is not defined), we set $\ell=+\infty$; if $J_1=J_2$, (i.e. $\{\rho_{J_2,n}\}_n$ is not defined), we set $\ell=0$. We distinguish between the cases $\ell\in (1,+\infty]$,  $\ell =1$ and $\ell\in [0,1)$.

\EMPH{Case 2a: $\ell>1$}
In particular, $J_0=J_1$ or $J_0<J_1<J_2$ and for large $n$, 
\begin{equation}
\label{grand_rhoJ2}
\rho_{J_2,n}>\rho_{J_1,n}.
\end{equation} 
Let
\begin{multline*}
 (u_{0,n},u_{1,n})=(v_{\lin}(t_n),\partial_tv_{\lin}(t_n))+\sum_{j=1}^{J_1}\left(\frac{\iota_j}{\lambda_{j,n}^{1/2}}W\left(\frac{x}{\lambda_{j,n}},0\right),0\right)+\sum_{j={J_1+1}}^{J_2} \vecc{\tU}_{\lin,n}^j(0)+\left(w_{0,n}^{J_2},w_{1,n}^{J_2}\right).
\end{multline*}
Note that for $J_0+1\leq j\leq J_1$ we have, if $|x|> \rho_{j,n}$
\begin{equation}
 \label{D14}
\vecc{U}_{\lin,n}^j(0,x)=\left(\frac{\iota_j}{\lambda_{j,n}^{1/2}}W\left(\frac{x}{\lambda_{j,n}}\right),0\right).
\end{equation} 
Thus, by \eqref{D10} and \eqref{grand_rhoJ2}, the equality \eqref{D14} holds for any $x$ such that $|x|>\rho_{J_2,n}$. As a consequence, $(u_{0,n}(x),u_{1,n}(x))= (u(t_n,x),\partial_tu(t_n,x))$ for $|x|\geq \rho_{J_2,n}$. Using that (by the definition of $\tU^{J_2}$ in Lemma \ref{L:channel_profile}):
\begin{equation}
 \label{D14'}
\int_{|x|\geq \rho_{J_2,n}+|t|} \left(\left|\nabla \tU_n^{J_2}(t,x)\right|^2+\left(\partial_t \tU_n^{J_2}(t,x)\right)^2\right)\geq \eta>0
\end{equation} 
holds for all $t\geq 0$ or for all $t\leq 0$, we see that we are exactly in the setting of Lemma \ref{L:channel}, which yields the desired contradiction.

\EMPH{Case 2b: $\ell=1$}
This case if very similar to case 2a. Let
$$ \left(H_{0,n}^{J_1},H_{1,n}^{J_1}\right)=\Psi_{\rho_{J_2,n}}\left(\frac{1}{\lambda_{J_1,n}^{1/2}}h_0^{J_1}\left(\frac{x}{\lambda_{J_1,n}}\right),\frac{1}{\lambda_{J_1,n}^{3/2}}h_1^{J_1}\left(\frac{x}{\lambda_{J_1,n}}\right)\right),$$
where $(h_0^{J_1},h_1^{J_1})$ is defined right after \eqref{D8} and the operator $\Psi_R$ in the beginning of Subsection \ref{SS:preliminaries}. Define
\begin{multline*}
 (u_{0,n},u_{1,n})=(v_{\lin}(t_n),\partial_t v_{\lin}(t_n))+\sum_{j=1}^{J_1} \left(\frac{1}{\lambda_{j,n}^{1/2}}W\left(\frac{x}{\lambda_{j,n}}\right),0\right)\\
+\sum_{j=J_1+1}^{J_2} \vecc{\tU}_{\lin,n}^j(0)+\left(H_{0,n}^{J_1},H_{1,n}^{J_1}\right)+\left(w_{0,n}^{J_2},w_{1,n}^{J_2}\right),
\end{multline*}
and note that $(u_{0,n}(x),u_{1,n}(x))=(u(t_n,x),\partial_tu(t_n,x))$ for $|x|\geq \rho_{J_2,n}$.
We have 
\begin{multline*}
\int_{\RR^3}\left(|\nabla H_{0,n}^{J_1}|^2+|H_{1,n}^{J_1}|^2\right)\,dx
=\int_{|x|\geq \rho_{J_2,n}} \left(\left|\nabla h_0^{J_1}\left(\frac{x}{\lambda_{J_1,n}}\right)\right|^2+\left(h_1^{J_1}\left(\frac{x}{\lambda_{J_1,n}}\right) \right)^2\right)\frac{dx}{\lambda_{J_1,n}^3}\\
= \int_{|y|\geq \frac{\rho_{J_2,n}}{\rho_{J_1,n}}\rho(h_0^{J_1},h_1^{J_1})} \left(\left|\nabla h_0^{J_1}(y)\right|^2+\left(h_1^{J_1}(y)\right)^2\right)\,dy\underset{n\to +\infty}{\longrightarrow}0
\end{multline*}
because $\ell=1$. Since \eqref{D14'} holds, the assumptions of Lemma \ref{L:channel} are again satisfied, yielding a contradiction.

\EMPH{Case 2c: $\ell\in [0,1)$}

Note that if $\rho\left(h_0^{J_1},h_1^{J_1}\right)>R_0$ ($R_0$ is defined by Proposition \ref{P:W+compct}) we can, using Lemma \ref{L:channel}, replace $U^{J_1}$ by a globally defined profile $\tU^{J_1}$ with the suitable energy channel property, and argue as in the preceding cases. In what follows, we reduce to this case using Proposition \ref{P:W+compct} \eqref{I:small_supp}.

Let $\cU^{J_1}$ be the solution obtained from $U^{J_1}$ by Proposition \ref{P:W+compct} \eqref{I:small_supp}. It has the following properties:
\begin{itemize}
 \item $\cU^{J_1}$ is defined for $t\in [-R_0,R_0]$;
\item there exists $R\in\left(0, \rho(h_0^{J_1},h_1^{J_1})\right)$ such that
\begin{equation}
 \label{D17}
\left(\cU^{J_1},\partial_t\cU^{J_1}\right)(0,x)=\left(U^{J_1},\partial_tU^{J_1}\right)(0,x)\text{ for }|x|\geq R.
\end{equation} 
\item the following holds for all $t\in [0,R_0]$ or for all $t\in [-R_0,0]$:
\begin{equation}
 \label{D18}
\rho\Big(\cU^{J_1}(t)-\iota_{J_1}W,\partial_t \cU^{J_1}(t)\Big)=\rho\left(h_0^{J_1},h_1^{J_1}\right)+|t|.
\end{equation} 
\end{itemize}
Assume to fix ideas that \eqref{D18} holds for all $t\in [0,R_0]$. The proof is the same in the other case. Let $\varphi\in C_0^{\infty}(\RR^3)$ radial, such that $\varphi(x)=1$ for $|x|\geq \frac{1}{2}$ and $\varphi(x)=0$ for $|x|\leq \frac 14$. Taking a larger $R$ if necessary, we can assume:
\begin{equation}
 \label{D19}
\ell \rho\left(h_0^{J_1},h_1^{J_1}\right)<R<\rho\left(h_0^{J_1},h_1^{J_1}\right).
\end{equation} 
Let
\begin{multline}
\label{D19'}
(\tu_{0,n},\tu_{1,n})=\vec{v}_{\lin}(t_n)+\varphi\left(\frac{x}{R\lambda_{J_{1},n}}\right)\left(\sum_{j=1}^{J_0} \left(\frac{\iota_j}{\lambda_{j,n}^{1/2}}W\left(\frac{x}{\lambda_{j,n}}\right),0\right)+\sum_{j=J_0+1}^{J_1-1} \vecc{U}_{\lin,n}^j(0)\right)\\
+\left(\cU^{J_1}_n(0),\partial_t\cU^{J_1}_n(0)\right)+\sum_{j=J_1+1}^{J_2} \vecc{\tU}^j_{\lin,n}(0)+\big(w_{0,n}^{J_2},w_{1,n}^{J_2}\big).
\end{multline}
We first claim that for large $n$:
\begin{equation}
 \label{D19''}
|x|>R\lambda_{J_1,n}\Longrightarrow (\tu_{0,n}(x),\tu_{1,n}(x))=(u(t_n,x),\partial_tu(t_n,x)),
\end{equation} 
and that for all $J\geq J_2$, 
\begin{multline}
\label{super_devt}
 (\tu_{0,n},\tu_{1,n})=\vec{v}_{\lin}(t_n)+\sum_{j\in \JJJ} \left(\frac{\iota_j}{\lambda_{j,n}^{1/2}}W\left(\frac{x}{\lambda_{j,n}}\right),0\right)\\
+(\cU_{n}^{J_1}(0),\partial_t\cU_{n}^{J_1}(0))+\sum_{j=J_1+1}^{J_2} \vecc{\tU}^j_{\lin,n}(0)+\sum_{j=J_2+1}^J \vecc{U}_{\lin,n}^j(0)+(\cw_{0,n}^J,\cw_{1,n}^J),
\end{multline}
 where
$$ \lim_{J\to \infty}\limsup_{n\to\infty} \left\|S(t)(\cw_{0,n}^J,\cw_{1,n}^J)\right\|_{L^8(\RR^4)}=0,$$
and $\JJJ$ is the set of indexes $j\in \{1,\ldots, J_0\}$ such that
$$ \lim_{n\to\infty} \frac{\lambda_{j,n}}{\lambda_{J_1,n}}=+\infty.$$

For large $n$, $\vecc{\tU}^j_{\lin,n}(0,x)=\vecc{U}^j_{\lin,n}(0,x)$ for $|x|>\rho_{J_2,n}$, and thus, by \eqref{D19} and the definition of $\ell$, for $|x|>R\lambda_{J_1,n}$. Combining with \eqref{D17}, we get \eqref{D19''}.

By the pseudo-orthogonality of the parameters, if $j\in \{1,\ldots,J_0\}\setminus \JJJ$, then $\lim_{n\to\infty}\frac{\lambda_{j,n}}{\lambda_{J_1,n}}=0$.
As a consequence
$$\lim_{n\to \infty}\left\|\varphi\left(\frac{x}{R\lambda_{J_1,n}}\right)\sum_{j=1}^{J_0} \left(\frac{\iota_j}{\lambda_{j,n}^{1/2}}W\left(\frac{x}{\lambda_{j,n}}\right),0\right)-\sum_{j\in \JJJ}\left(\frac{\iota_j}{\lambda_{j,n}^{1/2}}W\left(\frac{x}{\lambda_{j,n}}\right),0\right)\right\|_{\hdot\times L^2}=0.$$ 
Furthermore, if $j=J_0+1,\ldots, J_1-1$, then $\lim_{n\to \infty}\lambda_{j,n}/\lambda_{J_1,n}=0$ and thus
$$\lim_{n\to\infty}\left\|\varphi\left(\frac{x}{R\lambda_{J_1,n}}\right)\sum_{j=J_0+1}^{J_1-1}\vecc{U}_{\lin,n}^j(0)\right\|_{\hdot\times L^2}=0.$$
Thus \eqref{super_devt} follows from \eqref{D19'}.

Let $\tu_n$ be the solution of \eqref{CP} with initial data $(\tu_{0,n},\tu_{1,n})$. 
By \eqref{super_devt} and Proposition \ref{P:approx}, $\tu_{n}$ is defined on $[0,\lambda_{J_1,n}R_0]$ and:
\begin{multline}
 \label{D20}
\vecc{\tu}_n(\lambda_{J_1,n}R_0)=\vec{v}_{\lin}(\lambda_{J_1,n}R_0)+\sum_{j\in \JJJ} \left(\frac{\iota_j}{\lambda_{j,n}^{1/2}}W\left(\frac{x}{\lambda_{j,n}}\right),0\right)\\+\vecc{\cU}^{J_1}_n(\lambda_{J_1,n}R_0)+\sum_{j=J_1+1}^{J_2} \vecc{\tU}_{n}^j(\lambda_{J_1,n}R_0)+\sum_{j=J_2+1}^J \vecc{U}^j_n(\lambda_{J_1,n}R_0)+(\tw_{0,n}^J,\tw_{1,n}^J),
\end{multline}
where
\begin{equation}
 \label{D21}
\lim_{J\to\infty}\limsup_{n\to\infty} \left\|S(t)\left(\tw_{0,n}^J,\tw_{1,n}^J\right)\right\|_{L^8(\RR^4)}=0.
\end{equation} 
Note that
$$ \vecc{\cU}^{J_1}_n(\lambda_{J_1,n}R_0,x)=\left(\frac{1}{\lambda_{J_1,n}^{1/2}}\cU^{J_1}\left(R_0,\frac{x}{\lambda_{J_1,n}}\right),\frac{1}{\lambda_{J_1,n}^{3/2}}\partial_t\cU^{J_1}\left(R_0,\frac{x}{\lambda_{J_1,n}}\right)\right),$$
and recall that by Proposition \ref{P:W+compct}, 
$$ \rho\left(\cU^{J_1}(R_0)-\iota_{J_1}W,\partial_t\cU^{J_1}(R_0)\right)=R_0+\rho(h_0^{J_1},h_1^{J_1})>R_0.$$
Chose a positive $R'$ close to $\rho(h_0^{J_1},h_1^{J_1})$ such that $R<R'<\rho(h_0^{J_1},h_1^{J_1})$.
By Proposition \ref{P:W+compct} \eqref{I:big_supp}, there exists a globally defined solution $\tU^{J_1}$ of \eqref{CP}, scattering in both time directions such that, 
\begin{equation}
 \label{D22}
|x|>R_0+R'\Longrightarrow
\left(\cU^{J_1}(R_0,x),\partial_t\cU^{J_1}(R_0,x)\right)=\left(\tU^{J_1}(R_0,x),\partial_t\tU^{J_1}(R_0,x)\right)
\end{equation} 
and, there exists $\eta>0$ such that for all $t\geq 0$ or for all $t\leq 0$:
\begin{equation}
 \label{D23}
\int_{|x|\geq |t|+R_0+R'} \left[\left|\nabla \tU^{J_1}(R_0+|t|,x)\right|^2+\left(\partial_t \tU^{J_1}(R_0+|t|,x)\right)^2\right]\,dx\geq \eta.
\end{equation} 
Let
$$(u_{0,n},u_{1,n})=\vec{v}_{\lin}(t_n)+\sum_{j\in \JJJ} \left(\frac{\iota_j}{\lambda_{j,n}^{1/2}}W\left(\frac{x}{\lambda_{j,n}}\right),0\right)+\sum_{j=J_1}^{J_2}\vecc{\tU}_n^j(\lambda_{J_1,n}R_0)+(\tw_{0,n}^{J_2},\tw_{1,n}^{J_2}),$$
where $(\tw_{0,n}^{J_2},\tw_{1,n}^{J_2})$ is defined by \eqref{D20}. By finite speed of propagation, \eqref{D19''}, \eqref{D20} and \eqref{D22} we have
$$|x|>\lambda_{J_1,n}(R'+R_0)\Longrightarrow (u_{0,n},u_{1,n})(x)=\vecc{\tu}_{n}(\lambda_{J_1,n}R_0,x)=\vecc{\tu}(t_n+\lambda_{J_1,n}R_0,x).$$
Furthermore, by \eqref{D23}, the following holds for all $t\geq 0$ or for all $t\leq 0$:
\begin{equation*}
 \int_{|x|\geq \lambda_{J_1,n}(R'+R_0)+|t|} \left(\left|\nabla \tU_n^{J_1}(\lambda_{J_1,n}R_0+|t|,x)\right|^2+\left(\partial_t\tU_n^{J_1}(\lambda_{J_1,n}R_0+|t|,x)\right)^2\right)\,dx\geq \eta.
\end{equation*}
The assumptions of Lemma \ref{L:channel} are satisfied with $\rho_n=\lambda_{J_1,n}(R'+R_0)$ along the sequence of time $\{\ttt_n\}$, $\ttt_n=t_n+\lambda_{J_1,n}R_0$, yielding again a contradiction, which concludes the proof of Proposition \ref{P:sequence}. 
\qed
\subsection{Proof for all times}
\label{SS:all_times}
We now conclude the proof of the global case in Theorem \ref{T:main}. Let $u$ be a solution of \eqref{CP} such that $T_+(u)=+\infty$, and let $v_{\lin}$ be given by Lemma \ref{L:vL}. By Proposition \ref{P:boundedness}, there exists a sequence $t_n\to+\infty$ such that $\left\{\vec{u}(t_n)\right\}_n$ is bounded in $\hdot\times L^2$. By Proposition \ref{P:sequence}, there exist $J\in \NN$, $\iota_1,\ldots,\iota_J\in \{\pm 1\}^J$ and sequences $\{\lambda_{j,n}\}_n$ with
$$0<\lambda_{1,n}\ll \ldots \ll\lambda_{J,n}\ll t_n$$
such that (after extraction of a subsequence)
\begin{equation}
 \label{B'35}
\lim_{n\to\infty}\left\| \vec{u}(t_n)-\vec{v}_{\lin}(t_n)-\sum_{j=1}^{J} \left(\frac{\iota_j}{\lambda_{j,n}^{1/2}}W\left(\frac{x}{\lambda_{j,n}}\right),0\right)\right\|_{\hdot\times L^2}=0.
\end{equation} 

\EMPH{Step 1. Convergence of the norms} We first show
\begin{align}
 \label{B'36}
\lim_{t\to +\infty} \left\|\nabla (u-v_{\lin})(t)\right\|_{L^2}^2&=J\|\nabla W\|_{L^2}^2\\
\label{B'37}
\lim_{t\to +\infty} \left\|\partial_t(u-v_{\lin})(t)\right\|_{L^2}^2&=0.
\end{align}
Indeed, by \eqref{B'35},
$$\lim_{n\to+\infty} \left\|\nabla(u-v_{\lin})(t_n)\right\|_{L^2}^2=J\|\nabla W\|^{2}_{L^2}.$$
By the intermediate value theorem, if \eqref{B'36} does not hold, there exists a sequence $t_n'\to +\infty$ and a small $\eps\neq 0$ such that
\begin{equation}
 \label{B'38}
\lim_{n\to+\infty} \left\|\nabla(u-v_{\lin})(t_n')\right\|_{L^2}^2=J\|\nabla W\|^{2}_{L^2}+\eps.
\end{equation} 
By Proposition \ref{P:sequence}, there exists a subsequence of $\{t_n'\}_n$ and $J'\in \NN$ such that
$$\lim_{n\to+\infty} \left\|\nabla(u-v_{\lin})(t_n')\right\|_{L^2}^2=J'\|\nabla W\|^{2}_{L^2},$$
contradicting \eqref{B'38}. This proves \eqref{B'36}. We omit the very close proof of \eqref{B'37}.

\EMPH{Step 2. Choice of the scaling} Define, for $j=1\ldots J$ and $t>0$ large,
$$ B_j:=(j-1) \|\nabla W\|_{L^2}^2+\int_{|x|\leq 1}|\nabla W(x)|^2\,dx$$
and
\begin{equation}
 \label{B'39}
\lambda_{j}(t):=\inf\left\{ \lambda>0\text{ s.t. } \int_{|x|\leq \lambda} \left|\nabla(u-v_{\lin})(t,x)\right|^2\,dx\geq B_j\right\}.
\end{equation} 
In this step we show that if $\theta_n\to+\infty$, there exists a subsequence of $\{\theta_n\}_n$ and $\iota_1',\ldots,\iota_J'\in \{\pm 1\}^J$ such that
\begin{gather}
 \label{B'40}
\lim_{n\to+\infty} \left\|\vec{u}(\theta_n)-\vec{v}_{\lin}(\theta_n)-\sum_{j=1}^J \left(\frac{\iota_j'}{\lambda_{j}^{1/2}(\theta_n)}W\left(\frac{\cdot}{\lambda_j(\theta_n)}\right),0\right)\right\|_{\hdot\times L^2}=0\\
\label{B'41}
\lambda_{1}(\theta_n)\ll \lambda_{2}(\theta_n)\ll \ldots \ll \lambda_J(\theta_n)\ll \theta_n.
\end{gather}
Indeed, we know by Proposition \ref{P:sequence} that there exists a subsequence of $\{\theta_n\}_n$, signs $\{\iota_{j}'\}_{j=1\ldots J}$ and sequences $\{\lambda_{j,n}'\}_n$, such that 
\begin{gather}
 \label{B'42}
\lim_{n\to\infty}\left\| \vec{u}(\theta_n)-\vec{v}_{\lin}(\theta_n)-\sum_{j=1}^{J} \left(\frac{\iota_j'}{(\lambda_{j,n}')^{1/2}}W\left(\frac{x}{\lambda_{j,n}'}\right),0\right)\right\|_{\hdot\times L^2}=0\\
\label{B'42'}
0<\lambda_{1,n}'\ll \ldots \ll \lambda_{J,n}'\ll \theta_n.
\end{gather} 
Let $j\in \{1,\ldots,J\}$. In view of \eqref{B'42}, if $r_0>0$,
$$\lim_{n\to\infty} \int_{|x|\leq r_0\lambda_{j,n}'} |\nabla(u-v_{\lin})(\theta_n,x)|^2\,dx=(j-1)\|\nabla W\|^2_{L^2}+\int_{|x|\leq r_0} |\nabla W|^2\,dx.$$
This shows that if $r_0<1$, $r_0\lambda_{j,n}'<\lambda_j(\theta_n)$ for large $n$, and if $r_0>1$, $r_0\lambda_{j,n}'>\lambda_{j}(\theta_n)$ for large $n$. Hence 
$$\lim_{n\to+\infty} \frac{\lambda_{j}(\theta_n)}{\lambda_{j,n}'}=1,$$
and \eqref{B'40} and \eqref{B'41} follow from \eqref{B'42} and \eqref{B'42'} 

\EMPH{Step 3. End of the proof} Let $\delta>0$ and $\III=(\alpha_1,\ldots,\alpha_J)\in \{-1,+1\}^{J}$. Define
$$ \AAA_{\III,\delta}:=\left\{ f\in \hdot,\; \exists \lambda_1,\ldots,\lambda_J>0\text{ s.t. } \left\|f-\sum_{j=1}^{J} \frac{\alpha_j}{\lambda_j^{1/2}}W\left(\frac{\cdot}{\lambda_j}\right)\right\|_{\hdot}+\sum_{j=1}^{J-1}\frac{\lambda_j}{\lambda_{j+1}}<\delta\right\}.$$
Then:
\begin{claim}
\label{C:distance}
 There exists a small $\delta_0>0$ such that if $\III,\III'\in \{-1,+1\}^{J}$ with $\III\neq \III'$, then 
$$\left(f\in \AAA_{\III,\delta_0}\text{ and } g\in \AAA_{\III',\delta_0}\right)\Longrightarrow\|f-g\|_{\hdot}>\delta_0.$$
\end{claim}
\begin{proof}
 If not, we obtain sequences $\{\lambda_{1,n}\}_n,\ldots,\{\lambda_{J,n}\}_n$, $\{\lambda'_{1,n}\}_n,\ldots,\{\lambda'_{J,n}\}_n$ such that
$$ \lambda_{1,n}\ll \ldots \ll \lambda_{J,n}\quad \text{and}\quad\lambda_{1,n}'\ll \ldots \ll \lambda_{J,n}'$$
and 
$$\lim_{n\to+\infty} \left\|\sum_{j=1}^J \frac{\alpha_j}{\lambda_{j,n}^{1/2}} W\left(\frac{x}{\lambda_{j,n}}\right)-\sum_{j=1}^{J} \frac{\alpha'_j}{\lambda_{j,n}'}W\left(\frac{x}{\lambda_{j,n}'}\right)\right\|_{\hdot}=0$$
which implies easily $\alpha_j=\alpha'_j$ for all $j$, i.e. $\III=\III'$, contradicting the assumptions. 
\end{proof}
Let $\delta_0$ be as in Claim \ref{C:distance}. From Step 2, there exists $t_0>0$ such that  
$$\forall t>t_0,\quad u(t)-v_{\lin}(t)\in \bigcup_{\III\in \{\pm 1\}^J} \AAA_{\III,\delta_0}.$$
By Claim \ref{C:distance} and the continuity of $u$ in $\hdot$, there exists $\III$ such that 
$$\forall t>t_0,\quad  u(t)-v_{\lin}(t)\in \AAA_{\III,\delta_0}.$$
Letting $\III=(\iota_1,\ldots,\iota_J)$, we get by Step 2 and an easy contradiction argument that
$$\lim_{t\to+\infty} \left\|\vec{u}(t)-\vec{v}_{\lin}(t)-\sum_{j=1}^J\left(\frac{\iota_j}{\lambda_{j}^{1/2}(t)} W\left(\frac{\cdot}{\lambda_j(t)}\right),0\right)\right\|_{\hdot\times L^2}=0,$$
which concludes the proof.

\section{Sketch of proof in the finite time blow-up case}
\label{S:blow-up}
This section is devoted to the finite time blow-up case in Theorem \ref{T:main}. Since it is very similar to the proof of the global case which  makes up the preceding section, we will only sketch it, highlighting two points where the proofs are different. 

Consider a solution $u$ of \eqref{CP} such that $T_+(u)<\infty$ which does not satisfy \eqref{BUp}. We must show the expansion \eqref{expansion_u_bup}. Assume without loss of generality that $T_+(u)=1$. In Subsection \ref{SS:outside_light_cone} (the analog of Subsection \ref{SS:free_wave}), we show that $u$ converges outside the light cone $\{|x|\leq 1-t\}$, i.e. we construct the regular part $(v_0,v_1)$ of the expansion \eqref{expansion_u_bup}. The short proof, based on the small data theory and finite speed of propagation, is standard. In Subsection \ref{SS:sequences2} (the analog of Subsection \ref{SS:sequences}), we state that the expansion \eqref{expansion_u_bup} holds along sequences of times. More precisely, as in Subsection \ref{SS:sequences}, this type of expansion holds (after extraction) along any sequence of times $t_n\to 1$ such that $\left\{\vec{u}(t_n)\right\}_n$ is bounded in $\hdot\times L^2$. We omit most of the proof, which is exactly the same as in the global case except for the contradiction by the energy channel argument (the analog of Lemma \ref{L:channel}) where  we give some details. The proof of the fact that the results of Subsection \ref{SS:sequences2} imply the full expansion \eqref{expansion_u_bup} follows almost word by word Subsection \ref{SS:all_times} and we also omit it.

\subsection{Convergence outside the light cone}
\label{SS:outside_light_cone}

\begin{lemma}
\label{L:v0}
Let $u$ be a radial solution of \eqref{CP} such that $T_+(u)=1$ and assume  \eqref{BUp} does not hold. Then there exists $(v_0,v_1)\in \hdot\times L^2$ such that
\begin{equation}
\label{CV_to_v0_v1}
\lim_{t\to 1} \int_{|x|\geq 1-t} \left|\nabla u(t,x)-\nabla v_0(x)\right|^2+\left(\partial_t u(t,x)-v_1(x)\right)^2\,dx=0. 
\end{equation} 
\end{lemma}
\begin{proof}
 The proof is very close to the one in \cite[Section 3]{DuKeMe11a}, we only need to check that the assumption that $\vec{u}$ is bounded in $\hdot\times L^2$ made in this article can be relaxed to the assumption that $\vec{u}$ is only bounded along a sequence of times.

Since \eqref{BUp} does not hold, there exists a sequence $t_n\to 1$ such that $\{\vec{u}(t_n)\}_n$ is bounded in $\hdot \times L^2$. After extraction of a subsequence, we can assume that there exists $(v_0,v_1)\in \hdot \times L^2$ such that
$$ \vec{u}(t_n)\xrightharpoonup[n\to\infty]{} (v_0,v_1)\text{ weakly in }\hdot\times L^2.$$
Let $x_0\in \RR^3$, and $\delta_0>0$ be a small parameter to be specified later. We distinguish two cases:
\begin{itemize}
 \item \emph{First case:} there exists $\eps>0$ and a subsequence of $\{t_n\}_n$ (still denoted by $\{t_n\}_n$) such that
$$ \lim_{n\to \infty} \int_{|x-x_0|<\eps} \left(|\nabla u(t_n,x)|^2+(u(t_n,x))^2+ (\partial_tu(t_n,x))^2\right)\,dx<\delta_0.$$
In this case, we say that $x_0$ is a regular point, and we show that 
\begin{equation}
 \label{CV_regular}
\lim_{t\to 1} \int_{|x-x_0|<\eps/2} \left(|\nabla u(t,x)-\nabla v_0(x)|^2+ (\partial_tu(t,x)-v_1(x))^2\right)\,dx=0.
\end{equation} 
Indeed, chose $n$ large, so that $|t_n-1|<\eps/2$ and 
$$ \int_{|x-x_0|<\eps} \left(|\nabla u(t_n,x)|^2+(u(t_n,x))^2+ (\partial_tu(t_n,x))^2\right)\,dx<2\delta_0.$$
Let $(\tu_{0},\tu_1)\in \hdot\times L^2$ such that
\begin{equation}
\label{regular1}
(\tu_0,\tu_1)(x)=\vec{u}(t_n,x)\text{ for }|x-x_0|<\eps/2 \quad \text{and}\quad \left\|(\tu_0,\tu_1)\right\|_{\hdot\times L^2}\leq C_0\delta_0 
\end{equation} 
($C_0$ is an absolut constant). Let $\tu$ be the solution of \eqref{CP} such that $(\tu(t_n),\partial_t\tu(t_n))=(\tu_0,\tu_1)$. Chosing $\delta_0>0$ small enough, the small data theory implies that $\tu$ is globally defined. As a consequence, 
\begin{equation*}
\lim_{t\to 1} \int_{|x-x_0|<\eps/2} \left(|\nabla (u-\tu)(t,x)|^2+(\partial_t (u-\tu)(t,x))^2\right)\,dx=0 
\end{equation*} 
(indeed by finite speed of propagation and \eqref{regular1}, the integrand is $0$ if $|t-t_n|\leq \eps/2$). By the definition of $(v_0,v_1)$ and uniqueness of the weak limit, we deduce \eqref{CV_regular}.
\item  \emph{Second case:} for all $\eps>0$,
$$ \liminf_{n\to \infty}\int_{|x-x_0|<\eps} \left(|\nabla u(t_n,x)|^2+(u(t_n,x))^2+ (\partial_tu(t_n,x))^2\right)\,dx\geq \delta_0.$$
In this case we say that $x_0$ is singular.
\end{itemize}
As $\{\vec{u}(t_n)\}_n$ is bounded, there is only a finite number of singular point. By the radial symmetry, $0$ is the only singular point, and the local convergence \eqref{CV_regular} holds for any $x_0\in \RR^3\setminus \{0\}$. By a similar proof than the proof of \eqref{CV_regular}, we can show that there exists  a large $M>0$ such that
\begin{equation}
\label{CV_regular2}
\lim_{t\to 1}\int_{|x|\geq M} \int \left(|\nabla u(t,x)-v_0(x)|^2+(\partial_tu(t,x)-v_1(x))^2\right)\,dx=0.
\end{equation}  
Combining with \eqref{CV_regular} we get that \eqref{CV_regular2} holds for any $M>0$. 
Let $v$ be the solution of \eqref{CP} with data $(v_0,v_1)$ at $t=1$. By \eqref{CV_regular2} and finite speed of propagation, we get that for $t<1$ close to $1$ and $|x|\geq 1-t$, $\vec{v}(t,x)=\vec{u}(t,x)$ . Hence \eqref{CV_to_v0_v1}.
\end{proof}
\subsection{Analysis along a sequence of times}
\label{SS:sequences2}
The analog of Proposition \ref{P:sequence} is the following:
\begin{prop}
\label{P:sequence2}
Let $u$ be a solution of \eqref{CP} such that $T_+(u)=1$.
 Assume that there exists $t_n\to 1$ such that $\left\{\vec{u}(t_n)\right\}_n$ is bounded in $\hdot\times L^2$, and let $(v_0,v_1)$ be given by Lemma \ref{L:vL}. Then, after extraction of subsequences in $n$, there exist $J\geq 1$, $\iota_1,\ldots,\iota_J\in\{\pm 1\}$ and sequences $\{\lambda_{j,n}\}_n$ with $0<\lambda_{1,n}\ll \ldots \ll \lambda_{J,n}\ll 1-t_n$ such that
\begin{equation}
\label{dev_sequence2}
\vec{u}(t_n)-(v_0,v_1)-\sum_{j=1}^J \left(\frac{\iota_j}{\lambda_{j,n}^{1/2}}W\left(\frac{x}{\lambda_{j,n}},0\right),0\right) \underset{n\to +\infty}{\longrightarrow} 0
\end{equation} 
in $\hdot\times L^2$. 
\end{prop}
The following lemma is the finite-time analog of Lemma \ref{L:channel}:
\begin{lemma}
 \label{L:channel2}
Let $u$ be as in Proposition \ref{P:sequence2}.
There exists no sequence $\{t_n\}_n$ such that $t_n\to 1$ with the following property. There exists a sequence of functions $\{(u_{0,n},u_{1,n})\}_n$, bounded in $\hdot\times L^2$, and a sequence $\{\rho_n\}$ of nonnegative numbers such that
\begin{equation}
 \label{F9}
|x|\geq \rho_n\Longrightarrow (u(t_n,x),\partial_tu(t_n,x))=(u_{0,n}(x),u_{1,n}(x)),
\end{equation} 
and there exists $J_0\in \NN$, $\iota_1,\ldots,\iota_{J}\in \{\pm 1\}$ such that $(u_{0,n},u_{1,n})$ has a profile decomposition of the following form:
\begin{equation}
\label{F10}
(u_{0,n},u_{1,n})=(v_0,v_1)+\sum_{j=1}^{J_0} \left(\frac{\iota_j}{\lambda_{j,n}^{1/2}}W\left(\frac{x}{\lambda_{j,n}}\right),0\right)+\sum_{j=J_0+1}^J \vecc{U}_{\lin,n}^j(0)+(w_{0,n}^J,w_{1,n}^J), 
\end{equation} 
where for all $j\geq J_0+1$, the nonlinear profile $U^j$ is globally defined and scatters in both time directions. Furthermore, there exists $\eps_0>0$ such that one of the following holds:
\begin{enumerate}
 \item \label{I:F12} there exists $j_0\geq J_0+1$ such that for all $t\geq 0$ or for all $t\leq 0$:
\begin{equation}
\label{F12}
\forall n \quad \int_{|x|\geq \rho_n+|t|} |\nabla U_n^{j_0}(t,x)|^2+(\partial_t U_n^{j_0}(t,x))^2\,dx \geq \eps_0
\end{equation} 
or
\item \label{I:F12'} for at least one sign $+$ or $-$,
\begin{equation}
 \label{F12'} \lim_{J\to+\infty}\lim_{n\to+\infty}  \inf_{\pm t\geq 0} \int_{|x|\geq \rho_n+|t|} |\nabla w_n^{J}(t,x)|^2+(\partial_t w_n^{J}(t,x))^2\,dx \geq \eps_0.
\end{equation} 
\end{enumerate}
\end{lemma}
Assuming Lemma \ref{L:channel2}, the proof of Proposition \ref{P:sequence2} is the same as the one of Proposition \ref{P:sequence}, replacing everywhere Lemma \ref{L:channel} by Lemma \ref{L:channel2}, $t\to +\infty$ by $t\to 1$ and $(v_{\lin}(t),\partial_tv_{\lin}(t))$ by $(v_0,v_1)$. We leave the details to the reader.
\begin{proof}[Proof of Lemma \ref{L:channel2}]
As in the proof of Lemma \ref{L:channel}, we argue by induction on $J_0$. The inductive step is the same than in Lemma \ref{L:channel} and we will only detail the case $J_0=0$. 

We denote by $v$ the solution of \eqref{CP} such that $\vec{v}(1)=(v_0,v_1)$. Using scaling and time translation, we can assume without loss of generality that $[0,2]$ is included in $I_{\max}(v)$. 

Let $u_n$ be the solution of \eqref{CP} with data $(u_{0,n},u_{1,n})$ at $t=0$. By Proposition \ref{P:approx}, $u_n$ is defined on $[-1,1]$ for large $n$. Furthermore,
\begin{equation}
 \label{F13}
\vec{u}_n(t,x)=\vec{v}(1+t,x)+\sum_{j=1}^J \vecc{U}_n^j(t,x)+\vecc{w}_n^J(t,x)+\vecc{r}_n^J(t,x),
\end{equation} 
where 
\begin{equation}
 \label{F14}
\lim_{J\to\infty}\limsup_{n\to\infty} \sup_{t\in [-1,1]} \left\|\vecc{r}_n^J(t,x)\right\|_{\hdot\times L^2}^2=0.
\end{equation} 
First assume that \eqref{F12} or \eqref{F12'} hold for all $t\geq 0$. Then by \eqref{F13} at $t=\frac{1-t_n}{2}$, \eqref{F12} or \eqref{F12'} and the orthogonality Claim \ref{C:ortho}, we get that for large $n$,
\begin{multline}
 \label{F15}
\int_{|x|\geq \rho_n+\frac{1-t_n}{2}} \left|\nabla u_n\left(\frac{1-t_n}{2},x\right)-\nabla v\left(\frac{1+t_n}{2},x\right)\right|^2\,dx\\
+\int_{|x|\geq \rho_n+\frac{1-t_n}{2}} \left(\partial_t u_n\left(\frac{1-t_n}{2},x\right)-\partial_t v\left(\frac{1+t_n}{2},x\right)\right)^2\,dx \geq \frac{\eps_0}{2}.
\end{multline} 
We have used that $\lim_n \frac{1-t_n}{2}=0$, and hence, by continuity of $v$ at $t=1$,
$$\lim_{n\to\infty} \left\|\vec{v}\left(\frac{1+t_n}{2}\right)-\vec{v}\left(1+\frac{1-t_n}{2}\right)\right\|_{\hdot\times L^2}=0.$$
By finite speed of propagation and \eqref{F9}, we deduce from \eqref{F15} that for large $n$,
\begin{equation}
 \label{F16}
\int_{|x|\geq \rho_n+\frac{1-t_n}{2}} \left|(\nabla u-\nabla v)\left(\frac{1+t_n}{2},x\right)\right|^2
+\left((\partial_t u-\partial_t v)\left(\frac{1+t_n}{2},x\right)\right)^2\,dx\geq \frac{\eps_0}{2}.
\end{equation} 
Since $u-v$ is supported in $\{|x|\leq 1-t\}$, and $1-\frac{1+t_n}{2}=\frac{1-t_n}{2}$, this is a contradiction.

Next, we assume that \eqref{F12} or \eqref{F12'} hold for all $t\leq 0$. By \eqref{F13} at $t=-t_n$, \eqref{F12} or \eqref{F12'} and Claim \ref{C:ortho}, we get that for large $n$:
\begin{equation*}
 \int_{|x|\geq \rho_n+t_n} \left|\nabla u_n(-t_n,x)-\nabla v(0,x)\right|^2+\left(\partial_t u_n(-t_n,x)-\partial_t v(0,x)\right)^2\,dx\geq\frac{\eps_0}{2}.
\end{equation*} 
Using again \eqref{F9} and finite speed of propagation, we deduce that for large $n$,
$$\int_{|x|>\rho_n+t_n} |\nabla u_0(x)-\nabla v(0,x)|^2+(u_1(x)-\partial_t v(0,x))^2\,dx\geq \frac{\eps_0}{2}.$$
Letting $n\to\infty$, and using that $(u_0-v(0,x),u_1-\partial_t v(0,x))$ is almost everywhere $0$ in the set $\{|x|>1\}$, we get again a contradiction, concluding the proof.
\end{proof}

\appendix
\section{Cauchy problem for the linearized equation}
\label{A:linearized}
In this appendix we prove Lemma \ref{L:CPCarlos} and Claim \ref{C:WV_OK}. 
\begin{proof}[Proof of Lemma \ref{L:CPCarlos}]
 Let $F_V(h)=5 V^{4}h+10V^3h^2+10V^2h^3+5Vh^4+h^5$. We want to solve the equation
$$ h(t)=S(t)(h_0,h_1)+\int_{0}^t \frac{\sin\left((t-s)\sqrt{-\Delta}\right)}{\sqrt{-\Delta}} F_V(h(s))\,ds$$
by fixed point. Define
$$L^p_IL^q=L^p(I,L^q(\RR^3)),\quad \|h\|_S=\|h\|_{L^8_IL^8},\quad \|h\|_{W}=\|h\|_{L^4_IL^4}.$$
For $a>0$, we let
$$ B_a=\left\{v\in L^8_IL^8\text{ s.t. } \|v\|_S\leq a \text{ and }\|D_x^{1/2}v\|_W\leq a\right\}$$
and
$$ \Phi_{(h_0,h_1)}(v)=S(t)(h_0,h_1)+\int_{0}^t \frac{\sin\left((t-s)\sqrt{-\Delta}\right)}{\sqrt{-\Delta}}F_V(v(s))\,ds.$$
We will show that if \eqref{V_small} and \eqref{h0_small} hold, we can chose $a>0$ so that
$$\Phi_{(h_0,h_1)}: B_a\to B_a$$
and is a contraction.
By the Strichartz inequality (see \cite{GiVe95}, \cite{LiSo95})
\begin{equation}
 \label{E1}
\|S(t)(h_0,h_1)\|_{S}+\|D^{1/2}S(t)(h_0,h_1)\|_{W}\leq C\delta,
\end{equation} 
and, for $t\in I$,
\begin{multline*}
\left\|\int_0^t \frac{\sin\left((t-s)\sqrt{-\Delta}\right)}{\sqrt{-\Delta}}F_V(v(s))\,ds\right\|_S+\left\|D^{1/2}_x \int_0^t \frac{\sin\left((t-s)\sqrt{-\Delta}\right)}{\sqrt{-\Delta}}F_V(v(s))\,ds\right\|_W\\ \leq C\left\|D_x^{1/2} F_V(v)\right\|_{L^{4/3}_IL^{4/3}}.
\end{multline*}
We estimate  $\left\|D_x^{1/2} F_V(v)\right\|_{L^{4/3}_IL^{4/3}}$ using the chain, Leibnitz rule \cite{KePoVe93}:
\begin{align}
 \label{E2}
\left\|D_x^{1/2}(v^5)\right\|_{L^{4/3}_IL^{4/3}}&\leq C\|v\|^4_S\|D^{1/2}_xv\|_W\\
\label{E3}
\left\|D_x^{1/2}(Vv^4)\right\|_{L^{4/3}_IL^{4/3}}&\leq C\|v^4\|_{L^2_IL^2}\|D^{1/2}_xV\|_{L^4_IL^4}+C \|V\|_{L^8_IL^8}\left\|D_x^{1/2}(v^4)\right\|_{L^{\frac 85}_IL^{\frac 85}}\\
\notag
&\qquad \leq C\delta \|v\|_S^{4}+C\delta \|v\|^3_S\left\|D^{1/2}_xv\right\|_W\\
\label{E4}
\left\|D_x^{1/2}\left(V^2v^3\right)\right\|_{L^{4/3}_IL^{4/3}}&\leq C\left\|v^{3}\right\|_{L^{\frac 83}_IL^{\frac 83}}\left\|D_x^{1/2}(V^2)\right\|_{L^{\frac 83}_IL^{\frac 83}}+C\left\|D_x^{1/2}v^3\right\|_{L^2_IL^2}\left\|V^2\right\|_{L^4_IL^4}\\
\notag
&\qquad \leq C\delta \|v\|^3_S+C\delta^2 \|v\|^2_S\left\|D^{1/2}_x v\right\|_W\\
\label{E5}
\left\|D_x^{1/2}\left(V^3v^2\right)\right\|_{L^{4/3}_IL^{4/3}}&\leq C\left\|v^2\right\|_{L^4_IL^4}\left\|D_x^{1/2}\left(V^3\right)\right\|_{L^2_IL^2}+C\left\|V^3\right\|_{L^{\frac 83}_IL^{\frac 83}}\left\|D_x^{1/2}\left(v^2\right)\right\|_{L^{\frac 83}_IL^{\frac 83}}\\
\notag
&\qquad \leq C\delta\|v\|^2_S+C\delta^3\|v\|_S\left\|D^{1/2}_x v\right\|_W\\
\label{E6}
\left\|D^{1/2}_x\left(V^4v\right)\right\|_{L^{4/3}_IL^{4/3}}&\leq C\left\|D_x^{1/2}(V^4)\right\|_{L^{\frac 85}_IL^{\frac 85}}\|v\|_S+C\left\|v^4\right\|_{L^{2}_IL^2}\left\|D_x^{1/2}v\right\|_{W}\\
\notag
&\qquad \leq C\delta \|v\|_{S}+C\delta^4 \left\|D_x^{1/2}v\right\|_W.
\end{align}
By \eqref{E1}, we need that for some large $C_0>0$:
$$ C_0\delta\leq a/2.$$
By \eqref{E2}, we need for some large $C_1>0$:
$$ C_1 a^4 \leq 1/2.$$
Finally, by \eqref{E3}, \eqref{E4}, \eqref{E5} and \eqref{E6} we need that for some large $C_2>0$
$$ C_2\delta\left(a^3+a^2+\delta a^2+a +\delta^2a+1+\delta^3\right)\leq 1/2.$$
Taking $a=2C_0\delta$, we see that the preceding conditions are satisfied for small $\delta$, which shows that $\Phi_{(h_0,h_1)}$ maps $B_a$ to $B_a$. The contraction argument is similar and we omit it.
\end{proof}
\begin{proof}[Proof of Claim \ref{C:WV_OK}]
 We will write $f(r)\approx g(r)$ if $f(r)/g(r)$ has a limit in $(0,+\infty)$ as $r\to \infty$. 

We have, for $k\in \NN\setminus \{0\}$,
\begin{equation}
 \label{E7} 
W^k(r)\approx \frac{1}{r^k},\quad \nabla (W^k)\approx \frac{1}{r^{k+1}}.
\end{equation} 
Thus, if $p\in [1,\infty)$,
\begin{equation}
 \label{E8} 
W^k\in L^p(\RR^3) \iff kp>3
\end{equation}
and in this case
\begin{equation}
 \label{E8'}
\int_{|x|\geq R}W^{kp}(x)\,dx \approx \frac{1}{R^{kp-3}},\quad \int_{|x|\leq R}W^{kp}(R)\,dx \approx \frac{1}{R^{kp-3}}.
\end{equation} 
Similarly, if $q\in [1,\infty)$, 
\begin{equation}
 \label{E9} 
\nabla (W^k)\in L^q(\RR^3) \iff (k+1)q >3,\quad \int_{|x|\geq R} \left|\nabla(W^k)\right|^q\,dx \approx \frac{1}{R^{(k+1)q-3}}.
\end{equation} 
Recall from \cite[Lemma 5]{NaPo09} the following interpolation inequalities
\begin{equation}
 \label{E10}
\left\|D^{1/2}_xf\right\|_{L^{\ell}(\RR^3)}\leq C\|f\|_{L^p(\RR^3)}^{1/2}\|\nabla f\|_{L^q(\RR^3)}^{1/2},\quad \frac{1}{\ell}=\frac{1}{2p}+\frac{1}{2q}.
\end{equation} 
By \eqref{E8}, $W\in L^{8}(\RR^3)$. By \eqref{E8}, \eqref{E9} and \eqref{E10} we get:
\begin{align*}
 (k=1,\; \ell=4,\; p=6,\; q=3)&\qquad D^{1/2}_x W\in L^4\\
(k=2,\;\ell=\frac{8}{3},\; p=4,\; q=2)&\qquad D_x^{1/2} (W^2)\in L^{\frac{8}{3}}\\
(k=3,\; \ell=2,\; p=4,\; q=\frac{4}{3})&\qquad D_x^{1/2}(W^3)\in L^2\\
(k=4,\; \ell=\frac 85,\; p= \frac 83,\; q=\frac 87)&\qquad D_x^{1/2}(W^4)\in L^{\frac 85}. 
\end{align*}
This shows point \eqref{I:W_OK} in the Claim.

To prove \eqref{I:V_OK}, we use the same values of $\ell$, $p$ and $q$ as before to show that for all $t$,
$$V(t) \in L^8,\quad D_x^{1/2}V(t)\in L^4,\quad D_x^{1/2}(V^2(t))\in L^{\frac{8}{3}},\quad D_x^{1/2}(V^3(t))\in L^2,\quad D_x^{1/2}(V^4(t))\in L^{\frac 85}.$$
Furthermore, by \eqref{E8'},
$$ \|V(t)\|_{L^8}^8\approx \frac{1}{(R_0+|t|)^{5}},\quad R_0+|t|\to +\infty,$$
and by \eqref{E8'}, \eqref{E9} and \eqref{E10} with the values of $\ell$, $p$, $q$ given above we get, for any $k=1,2,3,4$,
\begin{multline*}
 \left\|D_x^{1/2}(V^k_{R_0}(t))\right\|_{L^{\ell}}^{\ell}\lesssim \left\| V_{R_0}^k(t)\right\|_{L^p}^{\frac{\ell}{2}}\left\|D_x^{1/2}\left(V_{R_0}^k(t)\right)\right\|_{L^q}^{\frac{\ell}{2}}\\
\lesssim \frac{1}{\left((|t|+R_0)^{kp-3}\right)^{\frac{\ell}{2p}}}\times \frac{1}{\left((|t|+R_0)^{(k+1)q-3}\right)^{\frac{\ell}{2q}}}\lesssim \frac{1}{(|t|+R_0)^{k\ell+\frac{\ell}{2}-3}}.
\end{multline*}
Checking that in each case, $k\ell+\frac{\ell}{2}-3>1$, we get that for large $R_0$, $V$ satisfies \eqref{V_small}. The proof is complete.
\end{proof}

\section{Pseudo-orthogonality of the profiles}
\label{A:ortho}
In this appendix we prove Claim \ref{C:ortho}.

\EMPH{Step 1: reduction to solutions of the linear equation}
If $\lim_n\frac{\theta_n-t_{j,n}}{\lambda_{j,n}}=+\infty$ (respectively $-\infty$), then by \eqref{bounded_strichartz}, $U^j$ scatters forward in time (respectively backward in time), i.e. there exists a solution $V^j$ to the linear wave equation such that 
$$\lim_{t\to+\infty} \left\|\vecc{V}^{j}(t)-\vecc{U}^j(t)\right\|_{\hdot\times L^2}=0 \text{ (respectively }\lim_{t\to -\infty}...=0).$$
If $\frac{\theta_n-t_{j,n}}{\lambda_{j,n}}$, is bounded, we can always assume, after extraction, that it converges to a real number $t_0$, and we define $V^j_{\lin}$ as the solution of the linear wave equation \eqref{lin_CP} with data $\vecc{U}^j(t_0)$ at $t=t_0$. In both cases, $V^j_{\lin}$ satisfies
$$\lim_{n\to\infty} \left\|\vecc{U}_n^j(\theta_n)-\vecc{V}_{\lin,n}^j(\theta_n)\right\|_{\hdot\times L^2}=0,$$
where 
\begin{equation}
\label{B41} 
V_{\lin,n}^j(t,x)=\frac{1}{\lambda_{j,n}^{1/2}}V_{\lin}^j\left(\frac{t-t_{j,n}}{\lambda_{j,n}},\frac{x}{\lambda_{j,n}}\right).
\end{equation} 
Arguing similarly for the index $k$, we see that it is sufficient to prove \eqref{orthojk} and \eqref{ortho_wnJ} with the nonlinear profiles $U^j_{n}$ and $U^k_n$ replaced by the linear profiles $V^j_{\lin,n}$ and $V^k_{\lin,n}$. Replacing $t_{j,n}$ and $t_{k,n}$ by $t_{j,n}-\theta_n$ and $t_{k,n}-\theta_n$, and $w_n^J(t,x)$ by $w_n^J(t-\theta_n,x)$, we see that we can also assume $\theta_n=0$. Finally, we must prove:
\begin{align}
 \label{B7'}
 j\neq k\Longrightarrow \lim_{n\to \infty} \int_{\rho_n\leq |x|\leq \sigma_n} \left(\nabla V^j_{\lin,n}(0,x)\cdot \nabla V_{\lin,n}^k(0,x)+\partial_t V^j_{\lin,n}(0,x)\cdot \partial_t V_{\lin,n}^k(0,x)\right)\,dx=0\\
\label{B8'}
J\geq j\Longrightarrow \lim_{n\to \infty} \int_{\rho_n\leq |x|\leq \sigma_n} \left(\nabla V^j_{\lin,n}(0,x)\cdot \nabla w_n^J(0,x)+ \partial_t V^j_{\lin,n}(0,x)\cdot \partial_t w_n^J(0,x)\right)\,dx=0.
\end{align}

\EMPH{Step 2. Proof of \eqref{B7'}} As usual, we will use that a radial, finite energy solution $v$ of \eqref{lin_CP} satisfies, for some $f\in L^{2}_{\loc}(\RR)$ such that $\dot{f}\in L^2(\RR)$:
$$ rv(t,|r|)=f(t+r)-f(t-r),\quad (t,r)\in \RR^2.$$
Letting $w$ be an other radial solution of \eqref{lin_CP}, such that $rw(t,|r|)=g(t+r)-g(t-r)$, we obtain by a straightforward integration by parts:
\begin{multline}
 \label{B42}
\int_{\rho_n}^{\sigma_n} (\partial_r v(t,r)\partial_rw(t,r)+\partial_tv(t,r)\partial_tw(t,r))r^2dr\\
=
2\int_{\rho_n}^{\sigma_n}\left(\dot{f}(t+r)\dot{g}(t+r)+\dot{f}(t-r)\dot{g}(t-r)\right)\,dr+\rho_nv(t,\rho_n)w(t,\rho_n)-\sigma_nv(t,\sigma_n)w(t,\sigma_n).
\end{multline} 
Let $rV^{j,k}_{\lin}(t,|r|)=f^{j,k}(t+r)-f^{j,k}(t-r)$. Applying \eqref{B42} to $V_{\lin,n}^j$ and $V_{\lin,n}^k$, we get,
\begin{align}
\label{B43}
 \int_{\rho_n\leq |x|\leq \sigma_n} &\left(\nabla V^j_{\lin,n}(0,x)\cdot \nabla V_{\lin,n}^k(0,x)+\partial_t V^j_{\lin,n}(0,x)\cdot \partial_t V_{\lin,n}^k(0,x)\right)\,dx\\
\tag{$A_n$}
&\qquad= \int_{\rho_n}^{\sigma_n} \frac{1}{\lambda_{j,n}^{1/2}} \dot{f}^j\left(\frac{-t_{j,n}+r}{\lambda_{j,n}}\right)\frac{1}{\lambda_{k,n}^{1/2}} \dot{f}^k\left(\frac{-t_{k,n}+r}{\lambda_{k,n}}\right)\,dr\\
\tag{$B_n$}
&\qquad+\int_{\rho_n}^{\sigma_n} \frac{1}{\lambda_{j,n}^{1/2}} \dot{f}^j\left(\frac{-t_{j,n}-r}{\lambda_{j,n}}\right)\frac{1}{\lambda_{k,n}^{1/2}} \dot{f}^k\left(\frac{-t_{k,n}-r}{\lambda_{k,n}}\right)\,dr\\
\tag{$C_n$}
&\qquad-\frac{\sigma_n^{1/2}}{\lambda_{j,n}^{1/2}} V^j\left(\frac{-t_{j,n}}{\lambda_{j,n}},\frac{\sigma_n}{\lambda_{j,n}}\right)
\frac{\sigma_n^{1/2}}{\lambda_{k,n}^{1/2}} V^k\left(\frac{-t_{k,n}}{\lambda_{k,n}},\frac{\sigma_n}{\lambda_{k,n}}\right)
\\
\tag{$D_n$}
&\qquad+ \frac{\rho_n^{1/2}}{\lambda_{j,n}^{1/2}} V^j_{\lin}\left(\frac{-t_{j,n}}{\lambda_{j,n}},\frac{\rho_n}{\lambda_{j,n}}\right)
\frac{\rho_n^{1/2}}{\lambda_{k,n}^{1/2}} V^k_{\lin}\left(\frac{-t_{k,n}}{\lambda_{k,n}},\frac{\rho_n}{\lambda_{k,n}}\right)
\end{align}
By density, we can assume $V^{j,k}_{\lin}(0),\,\partial_tV^{j,k}_{\lin}(0)\in C_0^{\infty}(\RR^3)$. Then there exists a constant $C>0$ such that
\begin{equation}
 \label{B44}
|V^j_{\lin}(t,r)|+|V^{k}_{\lin}(t,r)|\leq \frac{C}{r+1+|t|}.
\end{equation} 
(this follows from the expression $V^j_{\lin}(t,r)=\frac{1}{r}\int_{t-r}^{t+r}\dot{f}^j(s)\,ds$ and the fact that $\dot{f}^j$ is bounded and compactly supported). 

From \eqref{B44}, we see that the term $(C_n)$ goes to zero as $n\to +\infty$ unless (after extraction of a subsequence) the sequences $\left\{\frac{-t_{j,n}}{\lambda_{j,n}}\right\}_n$ and $\left\{\frac{-t_{k,n}}{\lambda_{k,n}}\right\}_n$ converge in $\RR$, and the sequences $\left\{\frac{\sigma_n}{\lambda_{j,n}}\right\}_n$ and $\left\{\frac{\sigma_n}{\lambda_{k,n}}\right\}_n$ converge in $(0,+\infty)$. This is excluded by the pseudo-ortogonality of the sequences of parameters $\{(\lambda_{j,n},t_{j,n})\}_n$, $\{(\lambda_{k,n},t_{k,n})\}_n$. Thus $\lim_{n\to \infty}(C_n)=0$ and by the same proof $\lim_{n\to\infty}(D_n)=0$.

It remains to treat the terms $(A_n)$ and $(B_n)$. We will focus on $(A_n)$, the proof that $(B_n)$ goes to zero is similar. We distinguish two cases.
\begin{itemize}
 \item Assume $\lim_{n\to \infty} \frac{\lambda_{j,n}}{\lambda_{k,n}}\in \{0,+\infty\}$. Using that $\dot{f}^j$ and $\dot{f}^k$ are compactly supported, we see that the domain of integration in the integral defining $(A_n)$ has Lebesgue measure smaller than $C\min(\lambda_{j,n},\lambda_{k,n})$. Hence
$$ |(A_n)|\leq C\frac{\min(\lambda_{j,n},\lambda_{k,n})}{\lambda_{j,n}^{1/2}\lambda_{k,n}^{1/2}}\underset{n\to\infty}{\longrightarrow} 0.$$
\item If (after extraction) $\lim_{n\to \infty} \frac{\lambda_{j,n}}{\lambda_{k,n}}=\ell\in (0,+\infty)$, then we must have, by pseudo-ortho\-go\-na\-li\-ty $\lim_{n\to\infty} \frac{|t_{j,n}-t_{k,n}|}{\lambda_{j,n}}=+\infty$, which shows that the supports of the $j$ and the $k$ terms in $(A_n)$ are disjoint for large $n$, and thus that $A_n=0$ for large $n$. This concludes the proof of \eqref{B7'}.
\end{itemize}
\EMPH{Step 3. Proof of \eqref{B8'}}
In view of \eqref{B7'}, it is sufficient to show \eqref{B8'} for some large $J\geq j$. We write $w_n^J=\frac{1}{r}\left(g_n^J(t+r)-g_n^J(t-r)\right)$. First note that:
\begin{equation}
 \label{B45}
\forall j\leq J,\quad \dot{g}_n^J(\lambda_{j,n}r+t_{j,n})\xrightharpoonup[n\to\infty]{} 0 \text{ in }L^{2}(\RR,dr).
\end{equation} 
This follows easily from \eqref{weak_CV_wJ} and we omit the proof. We have:
\begin{multline}
 \label{B46}
\int_{\rho_n<|x|<\sigma_n} \nabla w_{0,n}^J(x)\cdot \nabla V_{\lin,n}^{J}(0)\,dx+\int_{\rho_n<|x|<\sigma_n} w_{1,n}^J(x)\cdot \partial_t V_{\lin,n}^{J}(0)\,dx =\\
\int_{\rho_n}^{\sigma_n} \dot{g}_n^J(r)\frac{1}{\lambda_{j,n}^{1/2}} \dot{f}^j\left(\frac{-t_{j,n}+r}{\lambda_{j,n}}\right)\,dr+\int_{\rho_n}^{\sigma_n} \dot{g}_n^J(-r)\frac{1}{\lambda_{j,n}^{1/2}} \dot{f}^j\left(\frac{-t_{j,n}-r}{\lambda_{j,n}}\right)\,dr\\
+\sigma_n w_{0,n}^J(\sigma_n) V_{0,n}^j(\sigma_n)-\rho_n w_{0,n}^J(\rho_n)V_{0,n}^j(\rho_n).
\end{multline} 
The map $f\mapsto f(1)$ is a bounded linear form on the space $\hdot_{\rad}(\RR^3)$ of radial $\hdot$ functions. Taking  $J$ large, we can assume
$$ \rho_n^{1/2}w_{0,n}^J(\rho_n\cdot) \xrightharpoonup[n\to\infty]{} 0 \text{ weakly in } \hdot$$
and thus
$$ \lim_{n\to \infty} \rho_n^{1/2} w_{0,n}^J(\rho_n)=0.$$
Thus for large $J$, the boundary terms in \eqref{B46} tend to $0$ as $n\to\infty$. Furthermore
\begin{equation}
 \label{B47}
\int_{\rho_n}^{\sigma_n} \dot{g}_n^J(r)\,\frac{1}{\lambda_{j,n}^{1/2}} \dot{f}^j\left(\frac{-t_n^j+r}{\lambda_{j,n}}\right)\,dr =\int_{\frac{\rho_n-t_{j,n}}{\lambda_{j,n}}}^{\frac{\sigma_n-t_{j,n}}{\lambda_{j,n}}} \lambda_{j,n}^{1/2} \dot{g}^J_n(\lambda_{j,n}r+t_{j,n}) \dot{f}^j(r)\,dr.
\end{equation}
We skip the proof of the following claim, which follows immediately from the dominated convergence theorem: 
\begin{claim}
 Let $I_n$ be a sequence of intervals of $\RR$, $\indic_{I_n}$ the characteristic function of $I_n$, and assume $u_n\xrightharpoonup{}0$ in $L^{2}(\RR)$. Then $\indic_{I_n}u_n\xrightharpoonup{}0$ in $L^{2}(\RR)$.
\end{claim}
Using \eqref{B45} and the claim, we see that the left-hand side in \eqref{B47} goes to $0$ as $n\to\infty$. Similarly, the other integral term in \eqref{B46} tends to $0$, which concludes the proof of Claim \ref{C:ortho}.\qed

\section{Energy channels for profiles}
\label{A:channel_profile}
In this appendix we prove Lemma \ref{L:channel_profile}.
We will need the following preliminary result to treat the case where \eqref{infinite_lim} holds:
\begin{claim}
\label{C:easy_channel} 
Let $u_{\lin}$ be a nonzero radial solution of the linear wave equation \eqref{lin_CP}. Then there exists a radial solution $\tu_{\lin}$ of \eqref{lin_CP} with arbitrarily small energy, and constants $t_0>0$, $\eta>0$ and $\rho\in \RR$ such that
\begin{gather}
 \label{A9}
\forall t\geq t_0,\;\forall |x|>\rho+t,\quad (\tu_{\lin},\partial_t \tu_{\lin})(t,x)=(u_{\lin},\partial_t u_{\lin})(t,x)\\
\label{A10}
\forall t\geq t_0,\quad \int_{|x|>\rho+t}|\nabla \tu_{\lin}(t,x)|^2+(\partial_t\tu_{\lin}(t,x))^2\,dx\geq \eta.
\end{gather}
\end{claim}
Let us postpone the proof of Claim \ref{C:easy_channel} and prove Lemma \ref{L:channel_profile}.

 First assume that $t_{j,n}=0$ for all $j$. Then by Proposition \ref{P:channel1} or Proposition \ref{P:W+compct}, case \eqref{I:big_supp}, there exists a solution $\tU^j$ of \eqref{CP}, globally defined and scattering in both time directions and positive numbers $R_j$, $\eta_j$ such that
$$ (\tU^j,\partial_t\tU^j)(0,x)=(U^j,\partial_t U^j)(0,x)\text{ if } |x|\geq R_j$$
and the following holds for all $t\geq 0$ or for all $t\leq 0$
$$ \int_{|x|\geq R_j+|t|} \left|\nabla \tU^j(t,x)\right|^2+(\partial_t \tU^j(t,x))^2\,dx\geq \eta_j.$$
In this case the conclusion of the lemma holds with $\rho_{j,n}=\lambda_{j,n}R_j$ and $\tU_{\lin}^j=S(t)\left(\tU^j(0),\partial_t\tU^j(0)\right)$.

Next, assume
$$ \lim_{n\to+\infty}\frac{-t_{j,n}}{\lambda_{j,n}}=+\infty.$$
(The case where this limit is $-\infty$ follows from the change of variable $t\mapsto -t$).
Let $\rho_j\in \RR$, $\eta_j>0$, $t_j>0$ and $\tU^j_{\lin}$ be given by Claim \ref{C:easy_channel}. Let $\tU^j$ be the solution of \eqref{CP} such that $T^+(\tU^j)=+\infty$ and
\begin{equation}
 \label{wave_op}
\lim_{t\to+\infty} \left\|\nabla (\tU^{j}-\tU^j_{\lin})(t,x)\right\|^2_{L^2}+\left\|\partial_t (\tU^{j}-\tU^j_{\lin})(t,x)\right\|^2_{L^2}=0.
\end{equation}
Taking a larger $t_j$ and a smaller $\eta_j>0$ if necessary, we can assume by \eqref{wave_op} and the small data theory
$$ \forall t\geq t_j,\quad \int_{|x|>\rho_j+t}|\nabla \tU^j(t,x)|^2+(\partial_t\tU^j(t,x))^2\,dx\geq \eta_j.$$
Using that $\tU^{j}_{\lin,n}(t,x)=\frac{1}{\lambda_{j,n}^{1/2}}\tU^j_{\lin}\left(\frac{t-t_{j,n}}{\lambda_{j,n}},\frac{x}{\lambda_{j,n}}\right)$ and the analoguous formula for $U^j_{\lin,n}$, we get (taking $n$ large, so that $-t_{j,n}/\lambda_{j,n}\geq t_j$),
$$\vecc{\tU}_{\lin,n}^j(0,x)=\vecc{U}_{\lin,n}^j(0,x)\text{ for } |x|>\rho_j\lambda_{j,n}-t_{j,n}$$
and 
$$ \forall t\geq 0,\quad \int_{|x|\geq \rho_j\lambda_{j,n}+t-t_{j,n}} |\nabla \tU^{j}_n(t,x)|^2+(\partial_t \tU^{j}_n(t,x))^2\,dx \geq \eta_j,$$
which yields the conclusion of the lemma with $\rho_n^j=\rho_j\lambda_{j,n}-t_{j,n}$ which is positive for large $n$.
\qed

\begin{proof}[Proof of Claim \ref{C:easy_channel}]
Using as usual that $U(t,r)=ru_{\lin}(t,|r|)$ is a solution of the transport equation $\partial_t^2U-\partial_r^2U=0$, which is odd in the variable $r$, we get
$$ ru_{\lin}(t,|r|)=f(t+r)-f(t-r),$$
where $f\in L^{2}_{\loc}(\RR)$, $\dot{f}\in L^2(\RR)$ and, for $r\geq 0$,
\begin{align}
 \label{def_f1}
\dot{f}(r)&=\frac{1}{2}\Big(\partial_r(ru_{\lin}(0,r))+\partial_t(ru_{\lin})(0,r)\Big)\\
\label{def_f2}
\dot{f}(-r)&=\frac{1}{2}\Big(-\partial_r(ru_{\lin}(0,r))+\partial_t(ru_{\lin})(0,r)\Big).
\end{align}
Let $t\geq 0$ and $\rho_0\in \RR$. A simple integration by parts yields:
\begin{multline*}
\int_{|x|\geq \rho_0+t}|\nabla u_{\lin}(t,x)|^2+(\partial_t u_{\lin}(t,x))^2\,dx\\=2\int_{\rho_0+t}^{+\infty}\left(\dot{f}^2(t+r)+\dot{f}^2(t-r)\right)\,dr+(\rho_0+t)\left(u_{\lin}(t,\rho_0+t)\right)^2. 
\end{multline*}
Hence:
\begin{multline}
 \label{A40}
\int_{|x|\geq \rho_0+t} |\nabla u_{\lin}(t,x)|^2+(\partial_tu_{\lin}(t,x))^2\,dx\\
=2\int_{\rho_0+2t}^{+\infty} \dot{f}^2(r)\,dr+2\int_{-\infty}^{-\rho_0}\dot{f}^2(r)\,dr+\frac{1}{\rho_0+t}\left(f(\rho_0+2t)-f(-\rho_0)\right)^2.
\end{multline} 
Let $\eps$ be a small positive number. Chose $\rho_0\in \RR$ such that
\begin{equation}
 \label{A40'}
2\int_{-\infty}^{-\rho_0}\dot{f}^2(r)\,dr=\eps.
\end{equation} 
We have, for $R_0>0$ large and $r\geq R_0$
\begin{equation*}
 |f(r)-f(R_0)|=\left|\int_{R_0}^r \dot{f}(s)\,ds\right|\leq \sqrt{r}\sqrt{\int_{R_0}^r \dot{f}^2(s)\,ds},
\end{equation*}
which shows (arguing similarly for negative $r$),
\begin{equation}
 \label{A41} 
\lim_{r\to\pm \infty} \frac{1}{r}f^2(r)=0.
\end{equation} 
By \eqref{A40}, \eqref{A40'} and \eqref{A41}, we get that there exists $t_0>0$ such that
$$\forall t\geq t_0,\quad \eps\leq \int_{|x|\geq \rho_0+t} |\nabla u|^2+(\partial_tu)^2\,dx\leq 2\eps.$$
Letting $\tu_{\lin}$ be the solution of the linear wave equation \eqref{lin_CP} with initial data $\Psi_{\rho_0+t_0}(u(t_0),\partial_tu(t_0))$ at $t=t_0$, we get the conclusion of the claim.
\end{proof}

\bibliographystyle{acm}
\bibliography{toto}
\end{document}